\documentclass{svmult}
\usepackage{amsmath, amssymb}

\def\HB {\hfill\break}
\def\div_h{\mathrm{div_h}\,}
\def\div{\mathrm{div}\,}
\def \R {\mathbb R}

\def\P{{\mathbb{P}}}

\def\DD{{\mathcal{D}}}

\begin{document}

\title*{Anisotropic Navier-Stokes equations in a
bounded cylindrical domain}

\author{Marius Paicu and Genevi\`{e}ve Raugel}
\institute{Univ Paris-Sud and CNRS, Laboratoire de
Math\'{e}matiques d'Orsay,
Orsay Cedex, F-91405, FRANCE}

\maketitle

\section{Introduction}

Navier-Stokes equations with anisotropic viscosity are classical in
geophysical fluid dynamics. Instead of choosing a classical viscosity
$-\nu (\partial_1^2 +\partial_2^2 +\partial_3^2)$ in the case of
three-dimensional fluids, meteorologists often modelize turbulent
flows by putting a viscosity of the form $-\nu_h (\partial_1^2
+\partial_2^2) - \nu_v \partial_3^2$, where $\nu_v$ is usually much
smaller than $\nu_h$ and thus can be neglected (see Chapter 4 of the book of
Pedlovsky \cite{Pedlovsky} for a detailed discussion). \HB
More precisely, in geophysical fluids, the rotation of the earth
plays a primordial role. This Coriolis force introduces a penalized
skew-symmetric term $\varepsilon^{-1} u\times e_3$ into the equations,
where $\varepsilon >0$ is the Rossby number and $e_3=(0,0,1)$ is the
unit vertical vector. This leads to an asymmetry between the horizontal and
vertical motions. By the Taylor-Proudman theorem (see \cite{Pedlovsky} and
\cite{Taylor}), the fluid tends to have a two-dimensional behavior,
far from the boundary of the domain.
When the fluid evolves between two parallel plates with homogeneous
Dirichlet boundary conditions, Ekman boundary layers of the form
$U_{BL}(x_1,x_2,\varepsilon^{-1}x_3)$ appear near the boundary. In
order to compensate the term $\varepsilon^{-1} U_{BL} \times e_3$
  by the term $-\nu_v \partial_3^2 U_{BL}$, we need to
impose that $\nu_v= \beta \varepsilon$, for $\beta >0$ (see
\cite{GrMa97} and also \cite{CDGG06}).

\vskip 2mm

When the fluid occupies the whole space, the Navier-Stokes equations
with vanishing or small vertical viscosity are as follows
\begin{equation}
\label{equ1}
\begin{split}
&\partial_t u+u\nabla u-\nu_h(\partial_1^2 + \partial_2^2) u - \nu_v
\partial_3^2 u=
-\nabla p  ~\mbox{ in } \R^3~, \quad t >0 \\
&\div u=0  ~\mbox{ in } \R^3 ~, \quad t >0 \\
&u|_{t=0}=u_0 ,
\end{split}
\end{equation}
where $\nu_h >0$ and $\nu_v \geq 0$ represent the horizontal and
vertical viscosities and,  where $u= (u_1,u_2,u_3)$ and $p$ are the
vector field of the velocities and the pressure respectively.
In the case of vanishing vertical viscosity,
the classical theory of the Navier-Stokes equations does not apply
and new difficulties arise. Some partial
$L^2$-energy estimates still hold for System \eqref{equ1}, but they
do not allow to pass to the limit and to obtain a weak solution like
in the well-known construction of weak Leray solutions. One cannot
either directly use the results of Fujita and Kato on the existence of strong
solutions. Of course, neglecting the horizontal viscosity and
requiring a lot of regularity on the initial data, one can prove the
local existence of strong solutions by working in the frame of
hyperbolic symmetric systems. Thus, new methods have to be developed.

In the case $\nu_v =0$, this system has been first studied by Chemin,
Desjardins Gallagher and Grenier in \cite{CDGG00}, who showed local
and global existence of solutions in anisotropic Sobolev spaces which
take into account this anisotropy.
  More
precisely, for $s \geq 0$, let us introduce the anisotropic Sobolev
spaces,
$$
H^{0,s}= \{ u \in L^2(\R^3)^3 \, | \,  \|u\|_{H^{0,s}}^2 =
\int_{\R^3} (1 +|\xi_3|^2)^s |\widehat{u}(\xi)|^2 d\xi < +\infty\}
$$
and $H^{1,s}=\{ u \in H^{0,s} \, | \,\partial_i u \in  H^{0,s}
\, , \, i=1,2\}$ (such anisotropic spaces have been introduced in
\cite{If99} for the study of the Navier-Stokes system in thin domains). \HB
In \cite{CDGG00}, the authors showed that, for any $s_0 >1/2$, and
any $u_0 \in H^{0,s_0}$, there exist $T>0$ and a local solution $u \in
L^{\infty}((0,T), H^{0,s_0}) \cap L^2((0,T), H^{1,s_0})$ of the
anisotropic Navier-Stokes equations \eqref{equ1}. If $\|u_0\|_{H^{0,s_0}}
\leq c\nu_h$, where $c>0$ is a small constant, then the solution is
global in time. In the same paper, the authors proved that there
exists at most one solution $u(t)$ of the equations \eqref{equ1} in
the space $L^{\infty}((0,T), H^{0,s}) \cap L^2((0,T), H^{1,s})$ for
$s > 3/2$. In these results, there was a gap between the regularity
required for the existence of solutions and the one required for the
uniqueness. This gap was filled by Iftimie (\cite{If02}) who showed
the uniqueness of the solution in the space
$L^{\infty}((0,T), H^{0,s}) \cap L^2((0,T), H^{1,s})$ for
$s > 1/2$. \HB
Like the classical Navier-Stokes equations, the system \eqref{equ1} on
the whole space $\R^3$ has a scaling. Indeed, if $u$ is a
solution of the equations \eqref{equ1} on a time interval $[0,T]$
with initial data $u_0$, then  $u_{\mu} (t,x)= \mu u(\mu^2 t, \mu x)$
is also a solution of \eqref{equ1} on the time interval $[0,
\mu^{-2}T]$, with initial data $\mu u_0(\mu x)$. This was one of the
motivations of Paicu for considering initial data in the scaling invariant
Besov space $\mathcal{B}^{0,1/2}$. In \cite{Pai05bis}, Paicu proved
the local existence and uniqueness of the solutions of \eqref{equ1},
for initial data $u_0 \in \mathcal{B}^{0,1/2}$. He showed global
existence of the solution when the initial data
in $\mathcal{B}^{0,1/2}$ are small, compared to the horizontal
viscosity $\nu_h$ (for further details, see \cite{Pai04,Pai05,
Pai05bis}).
Very recently, in \cite{CheminZhang}, Chemin and Zhang
introduced the scaling invariant Besov-Sobolev spaces
$\mathcal{B}^{-1/2,1/2}_4$ and showed existence of global solutions
when the initial data $u_0$ in $\mathcal{B}^{-1/2,1/2}_4$ are small
compared to the horizontal viscosity $\nu_h$. This result implies
global wellposedness of \eqref{equ1} with high oscillatory initial
data.
\vskip 1mm

Notice that, in all the above results as well as in this paper, one
of the key observations is that, in the various essential energy
estimates, the partial derivative $\partial_3$ appears only when
applied to the component $u_3$ in terms like $u_3 \partial_3 u_3$. 
Even if there is no vertical
viscosity and thus no smoothing in the vertical variable, the divergence-free
condition implies that $\partial_3 u_3$ is regular enough to get good
estimates of the nonlinear term.
\vskip 2mm

Considering the anisotropic Navier-Stokes equations on the whole
space $\R^3$ (or on the torus $\mathbf{T}^3$) instead of on a bounded
domain with boundary leads to some simplifications. For example, the Stokes
operator coincides with the operator $-\Delta$ on the space of smooth
divergence-free vectors fields. Also one can use Fourier transforms
and the Littlewood-Paley theory. \HB
One of the few papers considering
the anisotropic Navier-Stokes equations on a domain with a boundary
is the article of Iftimie and Planas \cite{IfPl}, who studied the
anisotropic equations \eqref{equ1} on a half-space $\mathcal{H}$,
supplemented with the boundary condition
\begin{equation}
\label{DH}
u_3 =0 ~\mbox{ in } \partial \mathcal{H} ~, \quad t >0.
\end{equation}
This system of equations can be reduced to the case of the whole space
$\R^3$.  Indeed, let $w$ be a solution of the anisotropic Navier-Stokes
equations on the half-space $\mathcal{H}$.  Extending the components
$w_1$ and $w_2$ to $\R^3$ by an even reflection and the third component
$w_3$ by an odd reflection with respect to the plane $x_3=0$, we
obtain a vector field $\tilde{w}$, which is a solution of the
equations \eqref{equ1} on the whole space $\R^3$. Conversely, if
$u(t)$ is a solution of the anisotropic Navier-Stokes equations on the
whole space $\R^3$ with initial data $u_0 \in H^{0,1}$ satisfying the 
condition $u_{0,3}=0$ on $\partial \mathcal{H}$, then the restriction
of $u(t)$ to $\mathcal{H}$ is a solution of the anisotropic
Navier-Stokes equations on $\mathcal{H}$, satisfying the condition
\eqref{DH}. For initial data $u_0 \in L^2(\mathcal{H})^3$,
satisfying the condition $\partial_3 u_0 \in L^2(\mathcal{H})^3$,
Iftimie and Planas showed that the solutions of
the anisotropic Navier-Stokes equations on $\mathcal{H}$ are limits
of solutions $u_{\varepsilon}$ of the Navier-Stokes equations on 
$\mathcal{H}$ with
Navier boundary conditions on the boundary $\partial \mathcal{H}$ and
viscosity term $-\nu_h (\partial_1^2 + \partial_2^2) u_{\varepsilon} 
- \varepsilon
\partial_3^2 u_{\varepsilon}$ on a time interval $(0,T_0)$, where
$T_0 >0$ is independent of $\varepsilon$. If the initial data $u_0$
are small with respect to the horizontal viscosity $\nu_h$, they
showed the convergence on the infinite time interval $(0,+\infty)$.
In what follows, in our study of the anisotropic Navier-Stokes
equations on a bounded domain, we are also going to introduce an
auxiliary Navier-Stokes system with viscosity
$-\nu_h (\partial_1^2 + \partial_2^2) u_{\varepsilon} - \varepsilon
\partial_3^2 u_{\varepsilon}$ (see System $(NS_{\varepsilon})$
below), but instead of considering Navier boundary conditions in the
vertical variable, we will choose periodic conditions.

\vskip 4mm
In this paper, we study the global and local existence and uniqueness
of solutions to the anisotropic Navier-Stokes equations on a
bounded product domain of the type $Q =\Omega \times (0,1)$, where
$\Omega$ is a smooth domain, with homogeneous Dirichlet boundary conditions
on the lateral boundary $\partial \Omega \times (0,1)$.  For sake of
simplicity, we assume that $\Omega$ is a star-shaped domain. We denote
$\Gamma_0=\Omega \times\{0\}$ and $\Gamma_1=\Omega\times\{1\}$ the
top and the bottom of $Q$. More precisely, we consider the
system of equations
\begin{equation}(NS_h)
\begin{cases}
\partial_t u+u\nabla u-\nu_h \Delta_h u=
-\nabla p  ~\mbox{ in } Q~, \quad t >0 \\
\div u=0  ~\mbox{ in } Q ~, \quad t >0 \\
u|_{\partial\Omega\times (0,1)}=0~, \quad t >0 \\
u_3|_{\Gamma_0\cup\Gamma_1}=0 ~, \quad t >0 \\
u|_{t=0}=u_0 \in H^{0,1}(Q),
\end{cases}
\end{equation}
where the operator $\Delta_h=\partial_{x_1}^2+\partial_{x_2}^2$
denotes the horizontal Laplacian and $\nu_h>0$ is the horizontal
viscosity. Here $u \equiv(u_1,u_2,u_3) \equiv (u_h,u_3)$ is the 
vector field of velocities and $p$
denotes the pressure term. To simplify the discussion, we suppose that
the forcing term $f$ vanishes (the case of a non vanishing
forcing term as well as the asymptotic behaviour in time of the
solutions of $(NS_h)$ will be studied in a subsequent paper). \HB
Since the viscosity is anisotropic, we
want to solve the above system in the anisotropic functional space
$$
\widetilde{H}^{0,1}(Q)=\{u\in L^2(Q)^3\, | \, \div u=0\, ;\, \gamma_n u=0
\mbox{ on }\partial Q\, ; \, \partial_3 u\in L^2(Q)^3\},
$$
where $\gamma_n u$ is the extension of
the normal trace $u \cdot n$
  to the space $H^{-1/2}(\partial Q)$. Since $u$ belongs to
$L^2(Q)^3$ and that $\div u =0$, $\gamma_n u$ is well defined and
belongs to $H^{-1/2}(\partial Q)$.
For later use, we also introduce the space
$$
H^{0,1}(\Omega \times (a,b) )= \{u \in L^2(\Omega \times (a,b))^3\, | 
\, \div u=0\, ; \,
\partial_3 u \in L^2(\Omega \times (a,b))^3\},
$$
where $-\infty < a <b <+\infty$. Clearly, $\widetilde{H}^{0,1}(Q)$ is
a closed subspace of $H^{0,1}(Q)$.
\vskip 1mm

Considering a bounded domain $Q$ with lateral Dirichlet boundary conditions
instead of working with periodic boundary conditions for example
introduces a new difficulty.  In particular, we have to justify that
these Dirichlet boundary conditions make sense.  Before stating
existence results of solutions $u(t)$ to the system $(NS_h)$, we
describe our strategy to solve this problem.  One way for solving
System $(NS_h)$ consists in adding an artificial viscosity term
$-\varepsilon \partial_3^2 u$ (where $\varepsilon >0$) to the first
equation in $(NS_h)$ and in replacing the initial data $u_0$ by more
regular data $u_0^{\varepsilon}$, that is, in solving the system
\begin{equation}(NS_{h,\varepsilon})
\begin{cases}
\partial_t u+u\nabla u-\nu_h \Delta_h u-\varepsilon\partial_{3}^2
u=-\nabla p  ~\mbox{ in } Q~, \quad t >0 \\
\div u=0 ~\mbox{ in } Q ~, \quad t >0 \\
u|_{\partial\Omega\times (0,1)}=0~, \quad t >0 \\
u_3|_{\Gamma_0\cup\Gamma_1}=0 ~, \quad t >0 \\
u|_{t=0}=u_0^{\varepsilon}~,
\end{cases}
\end{equation}
where the initial data $u_0^{\varepsilon}$ are chosen in the functional space
$$
{\tilde H}^1_0(Q)=\{u \in H^1(Q)^3\, | \, \div u=0\, ;\, 
u|_{\partial\Omega\times
(0,1)}=0\,; \, u_3=0 \mbox{ on } \Gamma_0 \cup \Gamma_1\},
$$
and are close to $u_0$. In what follows, we will actually consider a
sequence of positive numbers $\varepsilon_n$ converging to $0$, when
$n$ goes to infinity.  Thus, we will also choose a sequence of initial data
$u_0^{\epsilon_n} \in {\tilde H}^1_0(Q)$
converging to $u_0$ in the space $\widetilde{H}^{0,1}(Q)$ when $n$
goes to infinity.  A
choice of such a sequence is possible since the space ${\tilde
H}^1_0(Q)$
is dense in $\widetilde{H}^{0,1}(Q)$ (see Lemma \ref{dense} in the
next section). \HB
Notice that, for $\varepsilon >0$, we need to replace $u_0$ by more
regular initial data $u_0^\epsilon$ in order to be able to apply the Fujita-Kato
theorem.

However, the system of equations $(NS_{h,\varepsilon})$ is still not a
classical system. Indeed, we need to impose boundary conditions to the
horizontal part $u_h(t,x_h, x_3)$ on the top $x_3=1$ and the bottom
$x_3=0$ (that is on $\Gamma_0 \cup \Gamma_1$). As in \cite{IfPl}, we
could impose Navier-type boundary conditions to the horizontal part
on $\Gamma_0$ and $\Gamma_1$. Here we will take another path. We will
extend the velocity field $u$ by symmetry to the domain 
$\tilde{Q}=\Omega\times (-1,1)$
and then solve the equations $(NS_{h,\varepsilon})$ on the symmetrical
domain $\tilde{Q}$ by
imposing homogeneous Dirichlet boundary conditions on the lateral
boundary and periodic conditions in the vertical variable $x_3$.
More precisely, let $u$ be a vector in ${\tilde H}^1_0(Q)$. We extend  $u$
to a vector $\tilde{u} \equiv \Sigma u$ on $\tilde{Q} = \Omega\times (-1,1)$ by
setting
\begin{equation}
\label{symetrie}
\begin{split}
(\Sigma u)_i(x_h, -x_3) \equiv \tilde u_i(x_h,-x_3)=  &u_i(x_h, x_3), 
\quad  i=1,2,
\quad 0 \leq x_3\leq1\\
(\Sigma u)_3 (x_h, -x_3) = \tilde u_3 (x_h, -x_3)=& -u_3(x_h,x_3),
\quad 0 \leq x_3\leq 1 ,
\end{split}
\end{equation}
and $(\Sigma u) (x_h,x_3) \equiv \tilde u (x_h, x_3)=u(x_h, x_3)$ for 
$0\leq x_3\leq 1$. We
notice that the vector $\tilde u= \Sigma u$ belongs to the space
$H^1(\tilde Q)^3$.
We introduce the functional space
\begin{equation*}
\begin{split}
\tilde{V} \equiv  H^1_{0,per}(\tilde Q)=\{u\in H^1(\tilde Q)^3 \, | 
\, \div u=0\,
  ;\,  & u|_{\partial\Omega\times (0,1)}=0\, ; \cr
  &\, u(x_h, x_3) = u(x_h,x_3
  +2)\}.
  \end{split}
  \end{equation*}
  The vector $\tilde u =\Sigma u$ clearly belongs to the space 
$H^1_{0,per}(\tilde
  Q)$.  \HB
  We finally consider the problem
   \begin{equation}(NS_{\varepsilon})
\begin{cases}
\partial_t u_{\varepsilon}+u_{\varepsilon}\nabla u_{\varepsilon}
-\nu_h \Delta_h u_{\varepsilon}
-\varepsilon\partial_{x_3}^2
u_{\varepsilon}=-\nabla p_{\varepsilon}  ~\mbox{ in } \tilde{Q}~, \quad t >0 \\
\div u_{\varepsilon}=0  ~\mbox{ in } \tilde{Q}~, \quad t >0  \\
u_{\varepsilon}|_{\partial\Omega\times (-1,1)}=0 ~, \quad t >0 \\
u_{\varepsilon}(x_h,x_3)=u_{\varepsilon}(x_h,x_3+2) ~, \quad t >0 \\
u_{\varepsilon}|_{t=0} = u_{\varepsilon,0} \in H^1_{0,per}(\tilde Q).
\end{cases}
\end{equation}
According to the classical Fujita-Kato theorem (see \cite{FuKa}), for any
$u_{\varepsilon,0} \in \tilde{V}$, there exists a unique local strong
solution $u_{\varepsilon}(t) \in C^0([0, T_{\varepsilon}),
\tilde{V})$ of the Navier-Stokes equations $(NS_{\varepsilon})$.
Moreover, this solution is classical and belongs to
$C^0((0,T_{\varepsilon}), H^2(\tilde{Q})^3) \cap
C^1((0,T_{\varepsilon}), L^2(\tilde{Q})^3)$. If the time existence
interval is bounded, that is, if $T_{\varepsilon} < + \infty$, then
\begin{equation}
\label{explose}
\|u_{\varepsilon}(t)\|_{\tilde{V}} \rightarrow_{t \rightarrow
T_{\varepsilon}^{-}} +\infty.
\end{equation}
We next introduce the ``symmetry map" $S:  u \in \tilde{V} \mapsto Su
\in \tilde{V}$ defined as follows
\begin{equation*}
\begin{split}
&(Su)_i(x_h, -x_3) = u_i(x_h,x_3)~, \quad i=1,2, \cr
&(Su)_3(x_h, -x_3)=-u_3(x_h,x_3) .
\end{split}
\end{equation*}
We remark that, if $u_{\varepsilon}(t) \in C^0([0,T_{\varepsilon}),
\tilde{V})$ is a solution of the Navier-Stokes equations
$(NS_{\varepsilon})$, then $Su_{\varepsilon} \in  C^0([0,T_{\varepsilon}),
\tilde{V})$ is a solution of the equations $(NS_{\varepsilon})$ with
$Su_{\varepsilon}(0)= Su_{\varepsilon,0}$. If
$u_{\varepsilon,0} = \Sigma u_0$, where $u_0$
belongs to ${\tilde H}^1_0(Q)$, then, $Su_{\varepsilon,0}=
u_{\varepsilon,0}$ and by the above
uniqueness property, the solutions $Su_{\varepsilon}(t)$ and 
$u_{\varepsilon}(t)$
coincide. This implies in particular that
$u_{\varepsilon,3}(t)$ vanishes
on $\Gamma_0 \cup \Gamma_1$ for any $t \in [0,T_{\varepsilon})$.
Since $u_{\varepsilon}(t)=Su_{\varepsilon}(t)$ belongs to 
$C^0((0,T_{\varepsilon}),
H^2(\tilde{Q}))$, this also implies that $\partial_3 u_{\varepsilon,h}$
vanishes on $\Gamma_0 \cup \Gamma_1$ for any $t \in
(0,T_{\varepsilon})$ and for any $\varepsilon>0$.
\vskip 2mm

In what follows, we denote by $\nabla_h $ the gradient operator
in the horizontal direction, that is, the gradient with respect to
the variables $x_1$ and $x_2$.
To summarize, for any $u_0 \in \widetilde{H}^{0,1}(Q)$, we will construct a
(unique) local (respectively global) solution $\bar{u} \in
L^{\infty}((0, T_0),\widetilde{H}^{0,1}(Q))$, with $\nabla_h \bar{u} \in
L^{2}((0, T_0),H^{0,1}(Q))$ (respectively $\bar{u} \in
L^{\infty}((0, +\infty),\widetilde{H}^{0,1}(Q))$, with $\nabla_h 
\bar{u}$ in the
space $L^{2}_{loc}((0, +\infty),H^{0,1}(Q))$) by proceeding as
follows. We consider a (decreasing) sequence of positive numbers 
$\varepsilon_m$ converging
to zero and, using Lemma \ref{dense}, a sequence of initial data
$u_0^m \in {\tilde H}^1_0(Q)$ converging to $u_0$ in $\widetilde{H}^{0,1}(Q)$,
when $m$ goes to infinity. Then, for each $m$, we solve the
problem $(NS_{\varepsilon_m})$ with initial data $u_{\varepsilon_m,0} =
\Sigma u_0^m$. We thus obtain a unique local (respectively global)
solution $u_{\varepsilon_m}(t) \in C^0((0, T),\tilde{V})$ of the problem
$(NS_{\varepsilon_m})$ where $0 <T <+ \infty$ (respectively
$T=+\infty$). We show that this sequence $u_{\varepsilon_m}(t)$
is uniformly bounded in $L^{\infty}((0, T_0),H^{0,1}(\tilde{Q}))$ and that
the sequence $\nabla_h u_{\varepsilon_m}(t)$ is uniformly bounded in
$L^{2}((0, T_0),H^{0,1}(\tilde{Q}))$, where $T_0>0$ is independent of
$m$ (depending only on $u_0$). If the initial data are small enough,
we show the global existence of these solutions, again with bounds in
$L^{\infty}((0, +\infty),H^{0,1}(\tilde{Q}))$ and $L^{2}((0, 
+\infty),H^{0,1}(\tilde{Q}))$,
independent of $m$. Using these uniform bounds, we show that the
sequences $u_{\varepsilon_m}(t)$ and $\nabla_h u_{\varepsilon_m}(t)$
are Cauchy sequences in $L^{\infty}((0, T_0), L^2(\tilde{Q}))$ and
in $L^{2}((0, T_0), L^2(\tilde{Q}))$ respectively, which implies the existence
of the solution $\bar{u} \in
L^{\infty}((0, T_0),\widetilde{H}^{0,1}(Q))$ of $(NS_h)$, with 
$\nabla_h \bar{u} \in
L^{2}((0, T_0), H^{0,1}(Q))$. The uniqueness of the solution $\bar{u}$
is straightforward and is proved in the same way as the Cauchy property
of the sequence $u_{\varepsilon_m}(t)$.

This paper is organized as follows. In the second section, we
introduce several notations and spaces. We also prove auxiliary
results, which will be used in the next sections.
In the third section, we show global existence results under
various smallness assumptions. The
fourth section is devoted to the proof of the local existence of
solutions for general initial data.

  \section{Preliminaries and auxiliary results}

  In what follows, for any $u$ in $H^{0,1}(Q)$, we will often use the
  notation $\div_h u \equiv \div_h u_h= \partial_1 u_1 + \partial_2
  u_2$. This quantity is well defined since, by the divergence free
  condition, $\div_h u_h =- \partial_3 u_3$.

  We begin this section by proving the density of ${\tilde H}^1_0(Q)$ in
  the space $\widetilde{H}^{0,1}(Q)$. For sake of simplicity, we 
assume below that
  $\Omega$ is star-shaped. This hypothesis allows us to give a
  constructive proof of the density result. This density should be
  true even without this additional assumption.

  \begin{lemma} \label{dense}
  Let $\Omega$ be a smooth bounded domain, which is star shaped. Then the space
  ${\tilde H}^1_0(Q)$ is dense in the space $\widetilde{H}^{0,1}(Q)$.
  \end{lemma}

\begin{proof} Without loss of generality, we will assume that $\Omega$
is star shaped with respect to the origin $0$. \HB
Let $u$ be an element of $\widetilde{H}^{0,1}(Q)$. We denote by 
$u^*(x_h,x_3)$ the extension of
$u(x_h,x_3)$ by zero on ${\bf R}^2 \times (0,1)$, that is,
\begin{equation*}
\begin{split}
u^*(x_h,x_3) = &u(x_h,x_3)~, \quad \forall (x_h,x_3) \in \Omega \times
(0,1) \cr
u^*(x_h,x_3) = & 0~, \quad \forall (x_h,x_3) \in ({\bf R}^2 -
\Omega )\times (0,1) .
\end{split}
\end{equation*}
We notice that the property $\gamma_n u=0$ on $\partial \Omega \times
(0,1)$ implies that, for any $\Phi \in \DD({\bf R}^2 \times (0,1))$
with $\Phi_{|\overline{\Omega}}=\varphi$, we have
\begin{equation*}
\begin{split}
\langle \div u^*, \Phi\rangle_{\DD', \DD} = & - \langle u^*, \nabla
\Phi\rangle_{\DD', \DD}
=  - \int_{\Omega \times (-1,1)} u \cdot \nabla \varphi \, dx_hdx_3 \cr
= & \int_{\Omega \times (-1,1)} \div u \, \varphi \, dx_hdx_3 - \langle
\gamma_n u, \varphi_{| \partial \Omega \times
(0,1)}\rangle_{H^{-1/2},H^{1/2}} \cr
= & \int_{\Omega \times (-1,1)} \div u \, \varphi \,dx_hdx_3 ,
\end{split}
\end{equation*}
which implies that $\div u^*=0$. We also notice that $u^*$ and
$\partial_3 u^*$ belong to $L^2({\bf R}^2 \times (0,1))^3$ and that
thus $\div_h u^*_h=-\partial_3 u^*_3$ is in $L^2({\bf R}^2 \times (0,1))$. \HB
We next want to approximate the vector $u^*$ by a vector with compact
support in $\Omega \times [0,1]$. To this end, inspired by the remark
1.7 of Chapter I of \cite{Temam}, we introduce the
vector $u^*_{\lambda}$, for $\lambda >1$, defined by
\begin{equation}
\label{ulambda}
\begin{split}
u^*_{\lambda,i} (x_h,x_3) = & u^*_i(\lambda x_h , x_3)~, \quad i=1,2, \cr
u^*_{\lambda,3} (x_h,x_3) = & \lambda u^*_3(\lambda x_h , x_3)~.
\end{split}
\end{equation}
We remark that $u^*_{\lambda}$ and $\partial_3 u^*_{\lambda}$ belong to
$L^2({\bf R}^2 \times (0,1))^3$ and that $u^*_{\lambda,3}(x_h,0)=
u^*_{\lambda,3}(x_h,1)=0$. Moreover,
\begin{equation}
\label{divlambda}
\begin{split}
\div_h u^*_{\lambda} (x_h,x_3) = &\lambda (\div_h u^*)(\lambda
x_h,x_3)~, \quad \partial_3 u^*_{\lambda, 3}(x_h, x_3) = \lambda
(\partial_3 u^*_3)(\lambda x_h,x_3)~, \cr
\div u_{\lambda}^*(x_h, x_3)=& \lambda ((\div_h u^*)(\lambda
x_h,x_3) + \lambda (\partial_3 u^*_3)(\lambda x_h,x_3)) =0
\end{split}
\end{equation}
For any $\lambda >1$, the support of $u^*_{\lambda}$ is contained in
$(\frac{1}{\lambda} \bar{\Omega}) \times [0,1]$ and therefore is
  a compact strict subset of $\Omega \times [0,1]$. \HB
Furthermore, using the Lebesgue theorem of dominated convergence, one
easily shows that $u^*_{\lambda}$ converges to $u^*$ and thus to $u$ 
in $\widetilde{H}^{0,1}(Q)$,
when $\lambda$ converges to $1$. \HB

\noindent We next introduce a smooth bump function with compact support $\rho
\in \DD({\bf R}^2)$ such that
$$
\rho (x_h) \geq 0~, \quad \int_{{\bf R}^2} \rho (x_h) dx_h =1 .
$$
For any small positive number $\eta$, we set
$$
\rho_{\eta}(x_h) = \frac{1}{\eta^2}
\rho(\frac{x_h}{\eta}) .
$$
It is well-known that $\rho_{\eta}(x_h)$ converges in the sense
of distributions to the Dirac distribution $\delta_{{\bf R}^2}$. For
any $\lambda >1$ and any $\eta >0$, where $\eta$ is
small with respect to $1-\lambda$, we consider the vector
$u^*_{\lambda,\eta}$, which is the ``horizontal convolution" of
$u^*_{\lambda}$ with $\rho_{\eta}$, that is,
\begin{equation}
\label{convol}
\begin{split}
u^*_{\lambda,\eta}(x_h,x_3) = &(\rho_{\eta} \star_{h} u^*_{\lambda})(x_h,x_3)
\cr
= & \int_{{\bf R}^2} \rho_{\eta}(y_h)u^*_{\lambda}(x_h-y_h,x_3)
dy_h .
\end{split}
\end{equation}
Since, for any $i=1,2,3$,
\begin{equation}
\label{difconvol}
\partial_i u^*_{\lambda,\eta} (x) \equiv \partial_i (\rho_{\eta} \star_{h}
u^*_{\lambda}) (x) = (\rho_{\eta} \star_{h} \partial_i
(u^*_{\lambda}))(x) ,
\end{equation}
it directly follows that, for any $(x_h,x_3) \in {\bf R}^2 \times
(0,1)$,
\begin{equation}
\label{divconvol}
\div u^*_{\lambda,\eta} = \rho_{\eta} \star_{h} (\div u^*_{\lambda})
=0 .
\end{equation}
Using the Young inequality $\|\rho_{\eta} \star_{h}
f(x_h)\|_{L^2({\bf R}^2)} \leq c \|\rho_{\eta} \|_{L^1({\bf
R}^2}\|f(x_h)\|_{L^2({\bf R}^2)} \leq \|f(x_h)\|_{L^2({\bf R}^2)}$,
we can write
\begin{equation}
\label{L2convol}
\begin{split}
\|u^*_{\lambda,\eta}\|_{L^{2}}^{2} = &\int_{0}^{1}( \int_{{\bf R}^2}
| u^*_{\lambda,\eta}(x_h,x_3)|^2 dx_h) dx_3 \leq c^2 \int_{0}^{1}
( \int_{{\bf R}^2} | u^*_{\lambda}(x_h,x_3)| ^2 dx_h) dx_3 \cr
= &c^2 \|u^*_{\lambda}\|_{L^{2}}^{2} .
\end{split}
\end{equation}
The properties \eqref{difconvol} and \eqref{L2convol} also imply that
\begin{equation}
\label{H01convol}
\|\partial_3 u^*_{\lambda,\eta}\|_{L^{2}}  \leq c \|\partial_3 
u^*_{\lambda}\|_{L^{2}} .
\end{equation}
We remark that $u^*_{\lambda,\eta}$ is a $C^{\infty}$-function in the
horizontal variable. Indeed, for any integers $k_1$, $k_2$, we have
$$
\frac{\partial^{k_1 + k_2} u^*_{\lambda,\eta}} {\partial x_1^{k_1} 
\partial x_2^{k_2}}
(x_h,x_3)  =
\int_{{\bf R}^2} \frac{\partial^{k_1 + k_2}\rho_{\eta}} {\partial x_1^{k_1}
\partial x_2^{k_2}}(x_h-y_h)u^*_{\lambda}(y_h,x_3)
dy_h .
$$
If $\eta >0$ is small with respect to $\lambda -1$, the support of
$u^*_{\lambda,\eta}(x_h,x_3) =(\rho_{\eta} \star_{h}
u^*_{\lambda})(x_h,x_3)$
is a compact set strictly contained in $\Omega \times [0,1]$. All
these properties imply in particular that $u^*_{\lambda,\eta}$
belongs to the Sobolev space $H^1(Q)$. We also check that $u^*_{\lambda,\eta}$
vanishes on  $\Gamma_0 \cup \Gamma_1$. Thus, $u^*_{\lambda,\eta}$
belongs to the space $\tilde{H}^{1}_{0}(Q)$, for $\eta >0$ small
enough with respect to $\lambda -1$.

\noindent For any fixed $\lambda >0$, one shows, like in the
classical case of convolutions in all the variables, that
\begin{equation}
\label{convconvol}
\begin{split}
\rho_{\eta} \star_{h} u^*_{\lambda} & \rightarrow_{\eta \rightarrow 0}
u^*_{\lambda} \hbox{ in }L^2(Q) \cr
\partial_3 (\rho_{\eta} \star_{h} u^*_{\lambda}) & \rightarrow_{\eta
\rightarrow 0} \partial_3 u^*_{\lambda} \hbox{ in }L^2(Q) .
\end{split}
\end{equation}
A quick proof of the first property of \eqref{convconvol} is as
follows. Let $w_n \in \DD(\bar{Q})^3$ be a sequence of smooth vectors 
converging to
$u^*_{\lambda}$ in $L^2(Q)^3$. Arguing as in \eqref{L2convol}, by
using the Young inequality, one proves that, for any positive number
$\delta$, there exists an integer $n_{\delta}$ such that, for $n \geq
n_{\delta}$, for any $\eta >0$,
$$
\|w_n - u^*_{\lambda}\|_{L^{2}} + \| \rho_{\eta} \star_{h} w_n
- \rho_{\eta} \star_{h} u^*_{\lambda}\|_{L^{2}} \leq \frac{\delta}{2} .
$$
It thus remains to show for instance that $\|w_{n_\delta} -\rho_{\eta}
\star_{h} w_{n_\delta}\|_{L^{2}}$ converges to $0$, as $\eta$
goes to $0$. Using the $C^1$-regularity of the vector
$w_{n_\delta}$ as well as the properties of $\rho_{\eta}$, one easily
shows that
$$
(\rho_{\eta} \star_{h} w_{n_\delta}) (x_h,x_3) \rightarrow_{\eta
\rightarrow 0}
w_{n_\delta} (x_h,x_3) \hbox{ a.e. }(x_h,x_3) .
$$
Moreover, by \eqref{L2convol},
$$
  \|w_{n_\delta} -\rho_{\eta} \star_{h} w_{n_\delta}\|_{L^{2}} \leq
  (c+1) \|w_{n_\delta}\|_{L^{2}} .
  $$
  These two properties imply, due to the Lebesgue theorem of dominated
  convergence, that $\|w_{n_\delta} -\rho_{\eta}
\star_{h} w_{n_\delta}\|_{L^{2}}$ converges to $0$, as $\eta$
goes to $0$, that is, there exists $\eta_0 >0$ such that, for
any $0 < \eta \leq \eta_0$,
$$
\|w_{n_\delta} -\rho_{\eta} \star_{h} w_{n_\delta}\|_{L^{2}} \leq
\frac{\delta}{2} .
$$
The first property in \eqref{convconvol} is thus
proved. The second property in \eqref{convconvol} is shown in the
same way.

\noindent Let finally $\lambda_n >1$ and $\eta_n>0$ be two sequences
converging to $1$ and $0$ respectively when $n$ goes to infinity. To
complete the proof of the lemma, it suffices to notice that, by a
diagonal procedure, one can extract two subsequences $\lambda_{n_k}$
and $\eta_{n_k}$ such that $u^*_{\lambda_{n_k},\eta_{n_k}}$ converges
to $u^*$ in $\widetilde{H}^{0,1}(Q)$ as $n_k$ goes to infinity. The 
lemma is thus
proved.

\end{proof}

\begin{remark} \label{H02}
We can also define spaces with higher regularity in the vertical
variable. For instance, let
$$
\widetilde{H}^{0,2}(Q)=\{u\in L^2(Q)^3\, | \, \div u=0\, ;\, \gamma_n u=0
\mbox{ on }\partial Q\, ; \, \partial_3^i u\in L^2(Q)^3 \, , \,
i=1,2 \}
$$
and
\begin{equation*}
\begin{split}
\widetilde{H}^{0,2}_0(Q)=\{u\in L^2(Q)^3\, | \, &\div u=0\, ;\, \gamma_n u=0
\mbox{ on }\partial Q\, ; \cr
&\, \partial_3 u_h = 0 \mbox{ on }\Gamma_0 \cup
\Gamma_1\,;
\, \partial_3^i u\in L^2(Q)^3 \, , \,
i=1,2 \} .
\end{split}
\end{equation*}
For later use, we also introduce the space
\begin{equation*}
\begin{split}
{H}^{0,2}(\Omega \times (a,b))=\{u\in L^2(\Omega \times (a,b))^3\, |
\, &\div u=0\, ; \cr
&\, \partial_3^i u\in L^2(\Omega \times (a,b))^3 \, , \,
i=1,2 \},
\end{split}
\end{equation*}
where $-\infty < a < b < +\infty$. \HB
Arguing as in the proof of Lemma \ref{dense}, one shows that, under
the same hypothesis,  ${\tilde H}^1_0(Q) \cap H^{2}(Q)$ is dense in
$\widetilde{H}^{0,2}(Q)$ and that
${\tilde H}^1_0(Q) \cap H^{2}(Q) \cap \widetilde{H}^{0,2}_0(Q)$ is dense in
$\widetilde{H}^{0,2}_0(Q)$
\end{remark}

Let $1 \leq p \leq +\infty$ and $1 \leq q \leq +\infty$.  We denote by
$L^q_vL^p_h(\tilde{Q})= L^q((-1,+1);L^{p}(\Omega))$ or simply
$L^q_vL^p_h$ the space of (classes of) functions $g$ such that
$\|g\|_{L^q_vL^p_h}= \big(\int_{-1}^{+1}(\int_{\Omega} |g(x_h, x_3)|^p
dx_h)^{q/p}dx_3 \big)^{1/q}$ is finite.  We point out that the order
of integration is important.  Of course, $L^q_vL^q_h$ is the
usual space $L^q(\tilde{Q})$ and the norm $\|g\|_{L^q_vL^q_h}$ is
denoted by $\|g\|_{L^q}$. Likewise we define the spaces
$L^q_vL^p_h(Q)= L^q((0,+1);L^{p}(\Omega))$.

\begin{lemma} \label{LestLpLq} The following anisotropic estimates
hold. \HB
1) For any function $g$ in $L^2(\tilde{Q})$ (with $\nabla_h g \in
L^2(\tilde{Q})$) satisfying
  homogeneous Dirichlet boundary conditions on the boundary
$ \partial \Omega \times (-1,+1)$, we have the estimate
\begin{equation}
\label{estLpLq1}
\|g\|_{L^2_v(L^4_h)} \leq C_0 \|g\|_{L^2}^{\frac 12}\|\nabla_h
g\|_{L^2}^{\frac 12},
\end{equation}
2) For any function $g$ in $L^2(\tilde{Q})$, with $\partial_3
g \in L^2(\tilde{Q})$, we have the estimate,
\begin{equation}
\label{estLpLq2}
\|g\|_{L^{\infty}_v(L^2_h)}\leq  C_0 \Big(\|g\|_{L^2}^{\frac
12}\|\partial_3
g\|_{L^2}^{\frac 12} + \|g\|_{L^2} \Big),
\end{equation}
where $C_0 >1$ is a constant independent of $g$.
\end{lemma}

\begin{proof} We first prove Inequality \eqref{estLpLq1}. Since $g$
vanishes on the lateral boundary, using the Gagliardo-Nirenberg
and the Poincar\'{e} inequalities in the horizontal variable and also 
the Cauchy-Schwartz
inequality in the vertical variable, we obtain,
\begin{equation*}
\begin{split}
\|g\|_{L^2_v(L^4_h)}^2& \leq  C \int_{-1}^{+1} \big[\big( \int_{\Omega}
| g(x_h,x_3)| ^2 dx_h \big)^{1/2} \big( \int_{\Omega}
| \nabla_h g(x_h,x_3)| ^2 dx_h \big)^{1/2} \big] dx_3 \cr
\leq  C_0^2  &\big(\int_{-1}^{+1} \int_{\Omega}
| g(x_h,x_3)| ^2 dx_h dx_3 \big)^{1/2} \big(\int_{-1}^{+1}  \int_{\Omega}
| \nabla_h g(x_h,x_3)| ^2 dx_h dx_3 \big)^{1/2} .
\end{split}
\end{equation*}
To prove Inequality \eqref{estLpLq2}, we first apply the Agmon inequality
in the vertical variable and then the Cauchy-Schwartz inequality in
the horizontal variable. We get,
\begin{equation*}
\begin{split}
\sup_{x_3 \in (-1,+1)} (\int_{\Omega}& | g(x_h,x_3)| ^2 dx_h)^{1/2}
\leq   (\int_{\Omega} \sup_{x_3 \in (-1,+1)}|g(x_h,x_3)| ^2 dx_h)^{1/2} \cr
\leq  & C \Big( \int_{\Omega} \big[ (\int_{-1}^{+1} | g(x_h,
x_3)| ^2dx_3)^{1/2}
(\int_{-1}^{+1} | \partial_3g(x_h, x_3)|^2dx_3 )^{1/2} \cr
&\hphantom{C \Big( \int_{\Omega} \big[ (\int_{-1}^{+1} g(x_h,
x_3)^2dx_3)^{1/2}}
+\int_{-1}^{+1} | g(x_h, x_3)| ^2dx_3 \big] dx_h \Big)^{1/2} \cr
\leq &C_0 \Big( \|g\|_{L^2}^{\frac 12}\|\partial_3
g\|_{L^2}^{\frac 12} + \|g\|_{L^2} \Big).
\end{split}
\end{equation*}
\end{proof}

The previous lemma allows to estimate the term $(u \nabla
u,u)_{H^{0,1}}$, which will often appear in the estimates given
below. More precisely, we can prove the following lemma.

\begin{lemma} \label{Lunablav}
There exists a positive constant $C_1$ such that, for
any smooth enough divergence-free vector field $u$ and any smooth
enough vector field $v$, the following estimate holds,
\begin{equation}
\label{unablav}
\begin{split}
|(u\nabla v,v)_{H^{0,1}}| \leq C_1 \big(\|u\|_{H^{0,1}}^{
1/2}\|\nabla_h u\|_{H^{0,1}}^{1/2} & \|v\|_{H^{0,1}}^{
1/2}\|\nabla_h v\|^{3/2}_{H^{0,1}} \cr
&+\|\nabla_h
u\|_{H^{0,1}}\|v\|_{H^{0,1}}\|\nabla_h v\|_{H^{0,1}}\big).
\end{split}
\end{equation}
\end{lemma}

\begin{proof}
The proof of this lemma is very simple.  Integrating by parts and
using the divergence-free condition on $u$, we can write
\begin{equation}
\label{uvv1}
\begin{split}
(u\nabla v, v)_{H^{0,1}}&=(\partial_3(u\nabla v),\partial_3
v) = (\partial_3 u\nabla v,\partial_3 v) \cr
&= (\partial_3u_h\nabla_h
v,\partial_3 v) + (\partial_3 u_3\partial_3 v,\partial_3 v)\cr
&= (\partial_3u_h\nabla_h
v,\partial_3 v) -(\div_h u_h\partial_3 v,\partial_3 v).
\end{split}
\end{equation}
Applying Lemma \ref{LestLpLq}, we obtain the estimate,
\begin{equation}
\label{uvv2}
\begin{split}
|(\div_h u_h \partial_3v,\partial_3v)_{L^2}|& \leq
C\|\nabla_h u\|_{L^\infty_v(L^2_h)}\|\partial_3
v\|_{L^2_v(L^4_h)}^2\cr
\leq & C \Big(\|\nabla_h u\|_{L^2}^{1/2} \|\nabla_h \partial_3 u\|_{L^2}^{1/2}
  +\|\nabla_h u\|_{L^2} \Big) \|\partial_3 v\|_{L^2}\|\nabla_h
\partial_3 v\|_{L^2} \cr
  \leq & C\|\nabla_h u\|_{H^{0,1}}\|\partial_3v\|_{L^2}\|\nabla_h 
\partial_3 v\|_{L^2}
\end{split}
\end{equation}
Furthermore, using Lemma \ref{LestLpLq} once more, we get the estimate
\begin{equation}
\label{uvv3}
\begin{split}
|(\partial_3 u_h \nabla_h v,\partial_3v)|
\leq &C \|\partial_3u\|_{L^2_v(L^4_h)} \|\nabla_h
v\|_{L^\infty_v(L^2_h)}\|\partial_3 v\|_{L^2_v(L^4_h)} \cr
\leq &C \|\partial_3 u\|_{L^2}^{1/2}\|\nabla_h \partial_3 u\|_{L^2}^{1/2}
\|\partial_3 v\|_{L^2}^{1/2}\|\nabla_h \partial_3 v\|_{L^2}^{1/2} \cr
&\hphantom{C\|\partial_3 u\|}
\times\Big(\|\nabla_h v\|_{L^2}^{1/2}\|\nabla_h\partial_3 v\|_{L^2}^{1/2} +
\|\nabla_h v\|_{L^2} \Big) \cr
\leq & C \|\partial_3 u\|_{L^2}^{1/2}\|\nabla_h \partial_3 u\|_{L^2}^{1/2}
\|\partial_3 v\|_{L^2}^{1/2}\|\nabla_h \partial_3 v\|_{L^2}^{1/2}
\|\nabla_h v\|_{H^{0,1}}.
\end{split}
\end{equation}
The estimates \eqref{uvv1}, \eqref{uvv2} and \eqref{uvv3} imply the
lemma.
\end{proof}

\vskip 1mm
The next proposition shows that sequences of uniformly bounded (with respect to
$\varepsilon_m$) classical solutions of the equations
$(NS_{\varepsilon_m})$ converge to solutions
of the system $(NS_h)$ when $\varepsilon_m$ goes to zero. The same
type of proof implies the uniqueness of the solutions of System
$(NS_h)$. In order to state the result, we introduce the space
\begin{equation*}
H^{1,0}(\Omega \times (a,b))= \{u\in L^2(\Omega \times (a,b))^3\, |
\, \div u=0\, ;
\, \nabla_h u \in L^2(\Omega \times (a,b))^3 \},
\end{equation*}
where $-\infty < a < b < +\infty$.

\begin{proposition}\label{Cauchy}
1) Let $u_0 \in \widetilde{H}^{0,1}(Q)$ be given.  Let $\varepsilon_m >0 $ be a
(decreasing) sequence converging to $zero$ and
$u_0^m \in {\tilde H}^1_0(Q)$ be a sequence of initial data converging
to $u_0$ in $\widetilde{H}^{0,1}(Q)$, when $m$ goes to infinity. 
Assume that the
system $(NS_{\varepsilon_m})$, with initial data $\Sigma u_0^m$, 
admits a strong solution
$u_{\varepsilon_m}(t) \in C^0((0,T_0),\tilde{V})$ where $T_0$ does
not depend on $\varepsilon_m$ and that the sequences $u_{\varepsilon_m}(t)$
and $\nabla_h u_{\varepsilon_m}(t)$ are uniformly bounded in
$L^{\infty}((0, T_0),H^{0,1}(\tilde{Q}))$ and in
$L^{2}((0, T_0),H^{0,1}(\widetilde{Q}))$ respectively. Then, the sequence
$u_{\varepsilon_m}(t)$ converges
in $L^{\infty}((0,T_0), L^2(\tilde{Q})^3) \cap L^2((0, T_0),
H^{1,0}(\tilde{Q}))$ to a solution $u^*\in L^{\infty}((0,T_0),$ $
H^{0,1}(\tilde{Q}))$ of the problem $(NS_h)$, such that $\nabla_h u^*$ belongs
to $L^2((0, T_0), H^{0,1}(\tilde{Q}))$. In particular, the vector
field $u^*$ belongs to $L^{\infty}((0,T_0), \widetilde{H}^{0,1}(Q))$.
\HB
2) The problem $(NS_h)$ has at most one solution $u^*$ in $L^{\infty}((0,T_0),
H^{0,1}(\tilde{Q}))$ with $\nabla_hu^*$ in $L^2((0, T_0),
H^{0,1}(\tilde{Q}))$. \HB
\end{proposition}

\begin{proof}
We recall that $u_{\varepsilon_m}(t) \in C^0((0,T_0),\tilde{V})$ is a
classical solution of the equations
   \begin{equation}
  \label{epsm}
\begin{split}
& \partial_t u_{\varepsilon_m}+u_{\varepsilon_m}\nabla
u_{\varepsilon_m}-\nu_h \Delta_h u_{\varepsilon_m}
-\varepsilon_m\partial_{x_3}^2
u_{\varepsilon_m}=-\nabla p_{\varepsilon_m} \cr
& \div u_{\varepsilon_m}=0    \\
& u_{\varepsilon_m}|_{t=0} = \Sigma u_0^m .
\end{split}
\end{equation}
We first want to show that, under the hypotheses of the proposition, 
$u_{\varepsilon_m}(t)$
and $\nabla_h u_{\varepsilon_m}(t)$ are Cauchy sequences
in the spaces $L^{\infty}((0,T_0), L^2(\tilde{Q})^3)$ and in $L^2((0, T_0),
L^2(\tilde{Q})^3)$ respectively. \HB
In order to simplify the notation in the estimates below, we will
simply denote the vector $u_{\varepsilon_m}$ by $u_m$.
Let $m > k$. Since the sequence $\varepsilon_n$ is decreasing,
$\varepsilon_m <\varepsilon_k$. The vector $w_{m,k}=u_m-
u_k$ satisfies the equation
  \begin{equation*}
  \begin{split}
   \partial_t w_{m,k} -\nu_h \Delta_h w_{m,k} -\varepsilon_k
   \partial_3^2 w_{m,k} =  &(\varepsilon_m -\varepsilon_k) \partial_3^2
   u_m - w_{m,k} \nabla u_m  \cr
   & - u_k \nabla w_{m,k}
  -\nabla (p_{\varepsilon_m} -p_{\varepsilon_k}).
\end{split}
\end{equation*}
Taking the inner product in $L^2(\tilde{Q})^3$ of the previous
equality with $w_{m,k}$, we obtain the equality
  \begin{equation}
  \label{Cauchy1}
  \begin{split}
\frac{1}{2} \partial_t \|w_{m,k}\|_{L^{2}}^{2} + \nu_h \|\nabla_h
w_{m,k}\|_{L^{2}}^{2} + \varepsilon_k \|\partial_3 w_{m,k}\|_{L^{2}}^{2}
= & (\varepsilon_k -\varepsilon_m) (\partial_3 u_m,\partial_3
w_{m,k})_{L^2}\cr
& + B_1 +B_2,
\end{split}
\end{equation}
where
\begin{equation*}
  \begin{split}
B_1 & =- (w_{m,k,h} \nabla_h u_m, w_{m,k})_{L^2} \cr
B_2 & = -(w_{m,k,3} \partial_3 u_m, w_{m,k})_{L^2}.
\end{split}
\end{equation*}
Applying the H\"{o}lder and Young
inequalities and Lemma \ref{LestLpLq}, we estimate $B_1$ as follows,
  \begin{equation}
  \label{Cauchy2}
  \begin{split}
|B_1|  \leq & \|w_{m,k,h}\|_{L^2_v(L^4_h)}\|w_{m,k}\|_{L^2_v(L^4_h)}
\|\nabla_h u_m\|_{L^{\infty}_v(L^2_h)} \cr
\leq & \|w_{m,k}\|_{L^{2}} \|\nabla_h w_{m,k}\|_{L^{2}}\big(
\|\nabla_h u_m\|_{L^{2}} + \|\nabla_h u_m\|_{L^{2}}^{1/2}\|\nabla_h
\partial_3u_m\|_{L^{2}}^{1/2}\big) \cr
\leq & \frac{\nu_h}{4}\|\nabla_h w_{m,k}\|_{L^{2}}^2 +
\frac{4}{\nu_h} \|w_{m,k}\|_{L^{2}}^2\big(\|\nabla_h u_m\|_{L^{2}}^2 +
  \|\nabla_h \partial_3u_m\|_{L^{2}}^2 \big).
\end{split}
\end{equation}
Using the same arguments as above and also the fact that $\partial_3
w_{m,k,3}=- \div_h w_{m,k,h}$,  we can bound $B_2$ as follows,
  \begin{equation}
  \label{Cauchy3}
  \begin{split}
|B_2|  \leq & \|\partial_3 u_m\|_{L^2_v(L^4_h)} \|w_{m,k}\|_{L^2_v(L^4_h)}
\|w_{m,k,3}\|_{L^{\infty}_v(L^2_h)} \cr
\leq & \| \partial_3u_m\|_{L^{2}}^{1/2} \|\nabla_h
\partial_3u_m\|_{L^{2}}^{1/2}  \|w_{m,k}\|_{L^{2}}^{1/2} \|\nabla_h
w_{m,k}\|_{L^{2}}^{1/2}\cr
& \times \big( \|w_{m,k,3}\|_{L^{2}} + \|w_{m,k,3}\|_{L^{2}}^{1/2}
\|\partial_3w_{m,k,3}\|_{L^{2}}^{1/2}\big) \cr
\leq & \| \partial_3u_m\|_{L^{2}}^{1/2} \|\nabla_h
\partial_3u_m\|_{L^{2}}^{1/2}  \|w_{m,k}\|_{L^{2}}^{3/2}
\|\nabla_h w_{m,k}\|_{L^{2}}^{1/2} \cr
& + \| \partial_3u_m\|_{L^{2}}^{1/2} \|\nabla_h
\partial_3u_m\|_{L^{2}}^{1/2}  \|w_{m,k}\|_{L^{2}} \|\nabla_h
w_{m,k}\|_{L^{2}} \cr
\leq & \frac{\nu_h}{4}\|\nabla_h w_{m,k}\|_{L^{2}}^2 +
\frac{3}{2 \nu_h^{1/3}} \| \partial_3u_m\|_{L^{2}}^{2/3} \|\nabla_h
\partial_3u_m\|_{L^{2}}^{2/3}  \|w_{m,k}\|_{L^{2}}^{2} \cr
& + \frac{2}{\nu_h}  \| \partial_3u_m\|_{L^{2}}\|\nabla_h
\partial_3u_m\|_{L^{2}}  \|w_{m,k}\|_{L^{2}}^2.
\end{split}
\end{equation}
The estimates \eqref{Cauchy1}, \eqref{Cauchy2}, and \eqref{Cauchy3}
together with the Cauchy-Schwarz inequality imply that, for $t \in
[0,T_0]$,
\begin{equation}
  \label{Cauchy4}
  \begin{split}
\partial_t \|w_{m,k}\|_{L^{2}}^{2}& + \nu_h \|\nabla_h
w_{m,k}\|_{L^{2}}^{2} + (\varepsilon_k +\varepsilon_m) \|\partial_3
w_{m,k}\|_{L^{2}}^{2}\cr
\leq &(\varepsilon_k-\varepsilon_m) \|\partial_3 u_m\|_{L^{2}}^{2}
  + \frac{4}{\nu_h} \|w_{m,k}\|_{L^{2}}^2\big(\|\nabla_h u_m\|_{L^{2}}^2 +
  \|\nabla_h \partial_3u_m\|_{L^{2}}^2 \big)\cr
& + \frac{3}{2 \nu_h^{1/3}} \| \partial_3u_m\|_{L^{2}}^{2/3} \|\nabla_h
\partial_3u_m\|_{L^{2}}^{2/3}  \|w_{m,k}\|_{L^{2}}^{2} \cr
& + \frac{2}{\nu_h}  \| \partial_3u_m\|_{L^{2}}\|\nabla_h
\partial_3u_m\|_{L^{2}}  \|w_{m,k}\|_{L^{2}}^2.
\end{split}
\end{equation}
Integrating the inequality \eqref{Cauchy4} from $0$ to $t$,
and applying the Gronwall lemma, we obtain, for
$0< t \leq T_0$,
\begin{equation}
  \label{Cauchy5}
  \begin{split}
\|w_{m,k}&(t)\|_{L^2}^{2}  + \nu_h \int_{0}^{t}\|\nabla_h
w_{m,k}(s)\|_{L^2}^{2}ds + (\varepsilon_k +\varepsilon_m) 
\int_{0}^{t}\|\partial_3
w_{m,k}(s)\|_{L^2}^{2}ds \cr
  \leq \Big[&(\varepsilon_k-\varepsilon_m) \int_{0}^{T_0} \| \partial_3
u_m(s)\|_{L^2}^{2} ds + \|u_0^m - u_0^k\|_{L^2}^{2} \Big] \cr
&\times \exp  \big(\frac{c_0}{\nu_h} \int_{0}^{T_0} \big( \|\nabla_h 
u_m(s)\|_{L^{2}}^{2} +
\|\nabla_h \partial_3 u_m(s)\|_{L^{2}}^{2} + \|\partial_3
u_m(s)\|_{L^{2}}^{2} \big) ds \big) \cr
& \times \exp  \big(\frac{c_1}{\nu_h^{1/3}} (\int_{0}^{T_0}
\|\nabla_h \partial_3 u_m(s)\|_{L^{2}}^{2}ds)^{1/3} (\int_{0}^{T_0}
\|\partial_3 u_m(s)\|_{L^{2}} ds)^{2/3} \big),
\end{split}
\end{equation}
where $c_0$ and $c_1$ are two positive constants independent of $m$
and $k$. \HB
Since the sequences $u_{\varepsilon_m}(t)$
and $\nabla_h u_{\varepsilon_m}(t)$ are uniformly bounded in
$L^{\infty}((0, T_0),$ $H^{0,1}(\tilde{Q}))$ and in
$L^{2}((0, T_0),H^{0,1}(\widetilde{Q}))$ respectively,
the estimate \eqref{Cauchy5} implies that $u_{\varepsilon_m}$
and $\nabla_h u_{\varepsilon_m}$ are Cauchy sequences in
$L^{\infty}((0,T_0), L^2(\tilde{Q})^3)$ and $L^2((0, T_0),
L^2(\tilde{Q})^3)$ respectively. Thus $u_{\varepsilon_m}$
converges in $L^{\infty}((0,T_0), L^2(\tilde{Q})^3)$ $\cap
L^2((0, T_0), H^{1,0}(\tilde{Q}))$
to an element $u^*$ in this same space. Moreover, $u^*$ and $\nabla_h u^*$
are bounded in $L^{\infty}((0, T_0),$ $H^{0,1}(\tilde{Q}))$ and in
$L^{2}((0, T_0),H^{0,1}(\tilde{Q}))$ respectively. The convergence in
the sense of distributions of $u_{\varepsilon_m}$ to $u^*$ and the
divergence-free property of the sequence $u_{\varepsilon_m}$ imply
that $u^*$ is also divergence-free. Furthermore, one easily shows
that the restriction of $u^*$ to $Q$ is a weak solution of the system $(NS_h)$.
{} From the equality $\Sigma u_{\varepsilon_m}(t)
=u_{\varepsilon_m}(t)$, it follows that $ \Sigma u^*(t)
=u^*(t)$. In particular, $u^*_3(t)$ vanishes on $\Gamma_0 \cup
\Gamma_1$. Finally, we notice that, since $u_{\varepsilon_m}(t)
\in C^0((0,T_0), \tilde{V})$ converges in $L^2((0, T_0),
H^{1,0}(\tilde{Q}))$, $u^*(t)$ satisfies the homogeneous Dirichlet
boundary condition on the lateral boundary $\partial \Omega \times
(-1,1)$ for almost all $t \in (0,T_0)$. \HB
2) One proves the uniqueness of the solution $u^*$ in $L^{\infty}((0,T_0),
H^{0,1}(\tilde{Q}))$ with $\nabla_hu^*$ in $L^2((0, T_0),
H^{0,1}(\tilde{Q}))$ in the same way as the above Cauchy property.
\end{proof}

We now state the classical energy estimate which will be
widely used in the next sections.

\begin{lemma} \label{energie} Let $u_\varepsilon(t) \in C^0([0,T_0],
\tilde{V})$ be the classical
solution of the equations $(NS_{\varepsilon})$ with initial data
$u_{\varepsilon,0} \in \tilde{V}$. Then the following estimates are
satisfied, for any $t \in [0 ,T_0]$, for any $0 \leq t_0 \leq t$,
\begin{equation}
\label{energie1}
\begin{split}
&\|u_\varepsilon(t)\|_{L^2}^{2} \leq  \|u_\varepsilon(0)\|_{L^2}^{2}
\exp (- 2\nu_h \lambda_0^{-1} t),  \cr
&\nu_h \int_{t_0}^{t} \|\nabla_h u_\varepsilon(s)\|_{L^2}^{2} ds +
\varepsilon \int_{t_0}^{t} \|\partial_3 u_\varepsilon(s)\|_{L^2}^{2}
ds \leq  \frac{1}{2} \|u_\varepsilon(0)\|_{L^2}^{2}
\exp (- 2\nu_h \lambda_0^{-1} t_0),
\end{split}
\end{equation}
where $\lambda_0$ is the constant coming from the Poincar\'{e}
inequality.
\end{lemma}

\begin{proof} Since $u_{\varepsilon} \in C^0([0,T_0], \tilde{V})$ is
the classical solution of $(NS_{\varepsilon})$, we can take the inner
product in $L^2(\tilde{Q})^3$ of the first equation in
$(NS_{\varepsilon})$ with $u_{\varepsilon}$ and integrate by parts.
We thus obtain, for $0 \leq t \leq T_0$,
\begin{equation}
\label{enaux1}
\partial_t \|u_\varepsilon(t)\|_{L^2}^{2} +  2\nu_h
\|\nabla_h u_\varepsilon(t)\|_{L^2}^{2} + 2 \varepsilon \|\partial_3
u_\varepsilon(t)\|_{L^2}^{2} \leq 0.
\end{equation}
Since $u_\varepsilon$ satisfies homogeneous Dirichlet boundary
conditions on the lateral boundary, there exists a positive constant
$\lambda_0$ depending only on $\Omega$ such that,
\begin{equation}
\label{Poincare}
\| u_\varepsilon\|_{L^2}^{2} \leq \lambda_0 \| \nabla_h 
u_\varepsilon\|_{L^2}^{2}.
\end{equation}
The inequalities \eqref{enaux1} and \eqref{Poincare} imply that, for 
$0 \leq t \leq T_0$,
$$
  \partial_t \|u_\varepsilon(t)\|_{L^2}^{2} +  2\nu_h
\lambda_0^{-1} \|u_\varepsilon(t)\|_{L^2}^{2} \leq 0.
$$
Integrating the previous inequality and applying Gronwall Lemma, we
obtain the first inequality in \eqref{energie1}. Integrating now the
inequality \eqref{enaux1} from $t_0$ to $t$ and taking into account
the first estimate in \eqref{energie1}, we obtain obtain the second
estimate of \eqref{energie1}.
\end{proof}
\vskip 2mm

We continue  this section by an auxiliary proposition, which will be used
several times in the proof of global existence of solutions of the
system $(NS_{\varepsilon})$.

\begin{proposition} \label{gradh}
  Let $u_{\varepsilon} \in C^0([0,T_0), \tilde{V})$ be a classical 
solution of Problem
$(NS_{\varepsilon})$. Let $T_n <T_0$ be a sequence
converging to $T_0$ when $n$ goes to infinity. If
$u_{\varepsilon}(t)$ is uniformly bounded in $L^{\infty}((0,T_n),
H^{0,1}(\tilde{Q})) \cap L^2((0,T_n),H^{0,1}(\tilde{Q}))$ and if
$\nabla_h u_{\varepsilon}$ and  $\varepsilon \partial_3 u_{\varepsilon}$
are uniformly bounded in $L^2((0,T_n),H^{0,1}(\tilde{Q}))$
as $n$ goes to infinity, then $u_{\varepsilon}$ is uniformly bounded
in $C^0([0,T_n] ,\tilde{V})$ and the classical solution
$u_{\varepsilon}$ exists on a time interval $[0, T_{\varepsilon})$
where $T_{\varepsilon} >T_0$. In particular, if $T_n$ is a sequence 
which goes to
infinity when $n$ goes to infinity, then the classical solution
$u_{\varepsilon}$ exists globally.
\end{proposition}

\begin{proof} Let $u_{\varepsilon} \in C^0([0,T_0), \tilde{V})$ be a
(local) classical solution of Problem $(NS_{\varepsilon})$. In order
to prove the proposition, we have to show that $\nabla_h u_{\varepsilon}$
is uniformly bounded in $L^{\infty}((0,T_n),L^2(\tilde{Q})^3)$ as $n$
goes to infinity. Since
$u_{\varepsilon}$ is a classical solution, all the a priori estimates
made below can be justified. Let $P$ be the classical Leray
projection of $L^2(\tilde{Q})^3$ onto $\tilde{H}$, where
$$
\tilde{H} =\{u\in L^2(\tilde Q)^3 \, | \, \div u=0\,
  ;\,  \gamma_nu|_{\partial\Omega\times (0,1)}=0\, ;
  \, u(x_h, x_3) = u(x_h,x_3 +2)\}.
$$
Taking the inner product in
$L^2(\tilde{Q})$ of the first equation of $(NS_{\varepsilon})$ with
$-P\Delta_hu_{\varepsilon}$, we obtain the equality
$$
-(\partial_t u_{\varepsilon}, \Delta_h u_{\varepsilon}) + \nu_h \|P
\Delta_h u_{\varepsilon}\|_{L^{2}}^{2} + \varepsilon(\partial_3^2
u_{\varepsilon}, P\Delta_h u_{\varepsilon}) = -(u_{\varepsilon}\cdot
\nabla u_{\varepsilon}, P \Delta_hu_{\varepsilon}).
$$
We remark that, for $0 <t <T_0$, $\partial_t u_{\varepsilon}$,
$\partial_3 u_{\varepsilon}$ and
$\partial^2_3 u_{\varepsilon}$ vanish on the lateral boundary and
are periodic in $x_3$. Moreover, the divergence of $\partial^2_3
u_{\varepsilon}$ vanishes. These properties imply on the one hand that
$$
-\int_{\tilde{Q}} \partial_t u_{\varepsilon}\cdot \Delta_h
u_{\varepsilon} dx= \int_{\tilde{Q}} \partial_t \nabla_h
u_{\varepsilon} \cdot \nabla_h u_{\varepsilon} dx =
\frac{1}{2} \partial_t \|\nabla_h u_{\varepsilon}\|_{L^{2}}^{2} .
$$
On the other hand, we can write, for $0 <t <T_0$,
$$
\int_{\tilde{Q}} \partial_3^2 u_{\varepsilon} \cdot P \Delta_h
u_{\varepsilon}dx
= \int_{\tilde{Q}} \partial_3^2 u_{\varepsilon} \cdot \Delta_h
u_{\varepsilon}dx = - \int_{\tilde{Q}} \partial_3 u_{\varepsilon} 
\cdot \Delta_h
\partial_3 u_{\varepsilon}dx = \|\nabla_h \partial_3 
u_{\varepsilon}\|_{L^{2}}^{2}.
$$
The previous three equalities imply that, for $0 <t <T_0$,
\begin{equation}
\label{gradhaux1}
\partial_t \|\nabla_h u_{\varepsilon}\|_{L^{2}}^{2} +  \nu_h \|P
\Delta_h u_{\varepsilon}\|_{L^{2}}^{2} +
2 \varepsilon \|\nabla_h \partial_3 u_{\varepsilon}\|_{L^{2}}^{2}
\leq \frac{1}{\nu_h} \|u_{\varepsilon} \cdot \nabla 
u_{\varepsilon}\|_{L^{2}}^{2}.
\end{equation}
To estimate the term $ \|u_{\varepsilon} \cdot \nabla
u_{\varepsilon}\|_{L^{2}}^{2}$, we write
\begin{equation}
\label{gradhaux2}
\begin{split}
\|u_{\varepsilon} \cdot \nabla u_{\varepsilon}\|_{L^{2}}^{2} =&
\int_{\tilde{Q}} (u_{\varepsilon,h} \cdot \nabla_h u_{\varepsilon}
+u_{\varepsilon,3} \cdot \partial_3 u_{\varepsilon})^2 dx \cr
\leq &
\, 2  \|u_{\varepsilon,h} \cdot \nabla_h u_{\varepsilon}\|_{L^{2}}^{2}
+ 2 \|u_{\varepsilon,3} \cdot \partial_3 u_{\varepsilon}\|_{L^{2}}^{2}.
\end{split}
\end{equation}
It remains to bound both terms in the right hand side of the
inequality \eqref{gradhaux2}.
Using the Gagliardo Nirenberg and the Poincar\'{e} inequalities, we
can write
\begin{equation*}
\begin{split}
\|u_{\varepsilon,h} \cdot \nabla_h u_{\varepsilon}\|_{L^{2}}^{2}
\leq &\int_{-1}^{1} \|u_{\varepsilon}(.,x_3)\|_{L^4_h}^2 \| \nabla_h
u_{\varepsilon}\|_{L^4_h}^2 dx_3 \cr
\leq & c_0 \Big( \int_{-1}^{1} \|u_{\varepsilon}(.,x_3)\|_{L^2_h}\| \nabla_h
u_{\varepsilon}\|_{L^2_h}^2  \|D_h \nabla_h
u_{\varepsilon}(.,x_3)\|_{L^2_h}dx_3 \cr
& + \int_{-1}^{1} \|u_{\varepsilon}(.,x_3)\|_{L^2_h}\| \nabla_h
u_{\varepsilon}\|_{L^2_h}^3 dx_3 \Big)\cr
\leq & c_1 \Big(\|u_{\varepsilon}\|_{L^{\infty}_v(L^2_h)}
\|\nabla_h u_{\varepsilon}\|_{L^{\infty}_v(L^2_h)}
\|\nabla_h u_{\varepsilon}\|_{L^2} \|D_h\nabla_h
u_{\varepsilon}\|_{L^2} \cr
& + \|u_{\varepsilon}\|_{L^{\infty}_v(L^2_h)}
\|\nabla_h u_{\varepsilon}\|_{L^{\infty}_v(L^2_h)}
\|\nabla_h u_{\varepsilon}\|_{L^2}^2 \Big) .
\end{split}
\end{equation*}
Applying now Lemma \ref{LestLpLq} to the previous inequality, we
obtain,
\begin{equation}
\label{gradhaux3}
\begin{split}
\|u_{\varepsilon,h} \cdot \nabla_h u_{\varepsilon}\|_{L^{2}}^{2}
\leq & c_2 \Big(\|u_{\varepsilon}\|_{L^2} + \|\partial_3
u_{\varepsilon}\|_{L^2}^{1/2} \|u_{\varepsilon}\|_{L^2}^{1/2} \Big)
\cr
& \times \Big( \|\nabla_h u_{\varepsilon}\|_{L^2} +
\|\partial_3\nabla_h u_{\varepsilon}\|_{L^2}^{1/2}
\|\nabla_h u_{\varepsilon}\|_{L^2}^{1/2} \Big) \cr
& \times  \Big( \|\nabla_h
u_{\varepsilon}\|_{L^2} +\|D_h \nabla_h u_{\varepsilon}\|_{L^2} \Big)
\|\nabla_h u_{\varepsilon}\|_{L^2} .
\end{split}
\end{equation}
The classical regularity theorem for the stationary Stokes problem (see for
example \cite{ConstantinFoias}, \cite{SoloScadi} or \cite{Temam})
implies that there exists a positive constant $K_0(\varepsilon)$,
which could depend on $\varepsilon$, such that,
\begin{equation}
\label{gradhaux3bis}
\|D_h \nabla_h u_{\varepsilon}\|_{L^2} \leq K_0(\varepsilon) \big(
\|P  \Delta_h u_{\varepsilon}\|_{L^2} +
\frac{\varepsilon}{\nu_h} \|\partial_3^2
u_{\varepsilon}\|_{L^2}\big).
\end{equation}
Using the Young inequality $2 ab \leq a^2 + b^2$, we deduce from
\eqref{gradhaux3} and \eqref{gradhaux3bis} that
\begin{equation*}
\begin{split}
  \|u_{\varepsilon,h} \cdot \nabla_h   u_{\varepsilon} \|_{L^{2}}^{2}
\leq & 4 c_2 \Big(\|u_{\varepsilon}\|_{L^2} + \|\partial_3
u_{\varepsilon}\|_{L^2} \Big)
\Big( \|\nabla_h u_{\varepsilon}\|_{L^2} + \|\partial_3\nabla_h
u_{\varepsilon}\|_{L^2} \Big)  \cr
\times \Big( &\|\nabla_h
u_{\varepsilon}\|_{L^2} + K_0(\varepsilon) (\|P  \Delta_h 
u_{\varepsilon}\|_{L^2} +
\frac{\varepsilon}{\nu_h} \|\partial_3^2
u_{\varepsilon}\|_{L^2})\Big) \|\nabla_h u_{\varepsilon}\|_{L^2} ,
\end{split}
\end{equation*}
and also
\begin{equation}
\label{gradhaux4}
\begin{split}
\frac{2}{\nu_h} \|u_{\varepsilon,h} \cdot \nabla_h   u_{\varepsilon}
\|_{L^{2}}^{2}
\leq   \frac{\nu_h}{2} &\|P  \Delta_h u_{\varepsilon}\|_{L^2}^2
+\frac{\varepsilon^2}{\nu_h} \|\partial_3^2
u_{\varepsilon}\|_{L^2}^2 + \frac{\nu_h}{2} \|\nabla_h 
u_{\varepsilon}\|_{L^2}^2 \cr
  + c_5 &\frac{K_0(\varepsilon)^2 +1 }{\nu_h^3}
(\|u_{\varepsilon}\|_{L^2}^2 + \|\partial_3
u_{\varepsilon}\|_{L^2}^2 ) \cr
& \times
( \|\nabla_h u_{\varepsilon}\|_{L^2}^2 + \|\partial_3\nabla_h
u_{\varepsilon}\|_{L^2}^2) \|\nabla_h u_{\varepsilon}\|_{L^2}^2.
\end{split}
\end{equation}
Likewise, using the Gagliardo Nirenberg and the Poincar\'{e}
inequalities, we can write
\begin{equation*}
\begin{split}
\|u_{\varepsilon,3} \cdot \partial_3 u_{\varepsilon}\|_{L^{2}}^{2}
\leq c_1 \|u_{\varepsilon,3}\|_{L^{\infty}_v(L^2_h)}
\|\nabla_h u_{\varepsilon,3}\|_{L^{\infty}_v(L^2_h)}
\|\partial_3 u_{\varepsilon}\|_{L^2} \|\nabla_h \partial_3
u_{\varepsilon}\|_{L^2} ,
\end{split}
\end{equation*}
which implies, due to Lemma \ref{LestLpLq},
\begin{equation*}
\begin{split}
\|u_{\varepsilon,3} \cdot \partial_3 u_{\varepsilon}\|_{L^{2}}^{2}
   \leq &c_2 \Big(\|u_{\varepsilon}\|_{L^2} + \|\partial_3
u_{\varepsilon,3}\|_{L^2}^{1/2} \|u_{\varepsilon,3}\|_{L^2}^{1/2} \Big)
\cr
& \times \Big( \|\nabla_h u_{\varepsilon,3}\|_{L^2} +
\|\partial_3\nabla_h u_{\varepsilon,3}\|_{L^2}^{1/2}
\|\nabla_h u_{\varepsilon,3}\|_{L^2}^{1/2} \Big) \cr
& \times  \|\partial_3 u_{\varepsilon}\|_{L^2}
\|\nabla_h \partial_3u_{\varepsilon}\|_{L^2} .
\end{split}
\end{equation*}
Using the Young inequalities $ab \leq \frac{1}{2} a^2 + \frac{1}{2} 
b^2$ and $ab
\leq \frac{1}{4} a^4 + \frac{3}{4} b^{4/3}$, we
deduce from the previous inequality that
\begin{equation}
\label{gradhaux5}
\begin{split}
\frac{2}{\nu_h} \|u_{\varepsilon,3} \cdot \partial_3 
u_{\varepsilon}\|_{L^{2}}^{2}
   \leq \frac{c_6}{\nu_h} \Big(\|u_{\varepsilon}\|_{L^2} + \|\partial_3
u_{\varepsilon,3}\|_{L^2} \Big) &\|\partial_3  u_{\varepsilon}\|_{L^2}
\cr
\times &\big( \|\nabla_h u_{\varepsilon,3}\|_{L^2}^2 +
\|\partial_3\nabla_h u_{\varepsilon,3}\|_{L^2}^{2}
\big) .
\end{split}
\end{equation}
Finally, we deduce from the estimates \eqref{gradhaux1},
\eqref{gradhaux4} and, \eqref{gradhaux5} that, for $0 \leq t <T_0$,
\begin{equation}
\label{gradhaux6}
\begin{split}
\partial_t \|\nabla_h u_{\varepsilon}(t)\|_{L^{2}}^{2} + \frac{\nu_h}{2}
&\|P \Delta_h u_{\varepsilon}\|_{L^{2}}^{2} \cr
\leq &
\frac{\varepsilon^2}{\nu_h} \|\partial_3^2
u_{\varepsilon}(t)\|_{L^2}^2 + \frac{\nu_h}{2} \|\nabla_h
u_{\varepsilon}(t)\|_{L^2}^2 \cr
& + L_\varepsilon(u_{\varepsilon}(t))
\big(\frac{2c_6}{\nu_h} + c_5\frac{(K_0(\varepsilon)^2 +1)}
{\nu_h^3} \|\nabla_h u_{\varepsilon}(t)\|_{L^2}^2 \big) ,
\end{split}
\end{equation}
where
$$
L_{\varepsilon}(u_{\varepsilon}(t))= (\|u_{\varepsilon}(t)\|_{L^2}^2 
+ \|\partial_3
u_{\varepsilon}(t)\|_{L^2}^2) \,
( \|\nabla_h u_{\varepsilon}(t)\|_{L^2}^2 + \|\partial_3\nabla_h
u_{\varepsilon}(t)\|_{L^2}^2 ) .
$$
Integrating the inequality \eqref{gradhaux6} from $0$ to $T_n$, where
$0 < T_n <T_0$, we infer from \eqref{gradhaux6} that, for any $T_n$,
with $0 < T_n <T_0$,
\begin{equation}
\label{gradhaux6bis}
\begin{split}
\|\nabla_h u_{\varepsilon}(T_n)\|_{L^{2}}^{2} &+ \frac{\nu_h}{2}
\int_{0}^{T_n} \|P \Delta_h u_{\varepsilon}(s)\|_{L^{2}}^{2} ds \cr
\leq  \,
& \int_{0}^{T_n}
\big( \frac{\varepsilon^2}{\nu_h} \|\partial_3^2
u_{\varepsilon}(s)\|_{L^2}^2 + \frac{\nu_h}{2} \|\nabla_h
u_{\varepsilon}(s)\|_{L^2}^2
  + \frac{2c_6}{\nu_h} L_\varepsilon(u_{\varepsilon}(s))\big) ds \cr
   +&\|\nabla_h u_{\varepsilon}(0)\|_{L^{2}}^{2}
  + \frac{c_5(K_0(\varepsilon)^2 +1)}{\nu_h^3} \int_{0}^{T_n}
L_\varepsilon(u_{\varepsilon}(s)) \|\nabla_h
u_{\varepsilon}(s)\|_{L^2}^2 ds .
\end{split}
\end{equation}
Using Gronwall Lemma
and taking into account the hypotheses made on
$u_{\varepsilon}$, we deduce from \eqref{gradhaux6bis} that, for any $T_n$,
with $0 < T_n <T_0$,
\begin{equation}
\label{gradhaux7}
\begin{split}
\|\nabla_h u_{\varepsilon}(T_n)\|_{L^{2}}^{2} \leq &
\Big[\int_{0}^{T_n}
\big( \frac{\varepsilon^2}{\nu_h} \|\partial_3^2
u_{\varepsilon}(s)\|_{L^2}^2 + \frac{\nu_h}{2} \|\nabla_h
u_{\varepsilon}(s)\|_{L^2}^2
  + \frac{2c_6}{\nu_h} L_\varepsilon(u_{\varepsilon}(s))\big) ds \cr
  & +\|\nabla_h u_{\varepsilon}(0)\|_{L^{2}}^{2} \Big]
  \exp \Big(\frac{c_5(K_0(\varepsilon)^2 +1)}{\nu_h^3} \int_{0}^{T_n}
L_\varepsilon(u_{\varepsilon}(s)) ds\Big) \cr
\leq & \big[\|\nabla_h u_{\varepsilon}(0)\|_{L^{2}}^{2} + k_1 +
\frac{2c_6}{\nu_h} k_2 \big]   \exp \Big(\frac{c_5(K_0(\varepsilon)^2
+1)}{\nu_h^3} k_2\Big) ,
\end{split}
\end{equation}
where $k_1$ and $k_2$ are positive constants independent of $T_n$
($k_1$ and $k_2$ can depend on $\varepsilon$).
Thus the proposition is proved.
\end{proof}
\vskip 1mm

We end this section by giving an upper bound of the $H^{0,2}$-norm of
the solution
$u_{\varepsilon}(t)$ of the system $(NS_{\varepsilon})$ on any
subinterval of the maximal interval of existence, when the initial
data $u_{\varepsilon,0}$ belong to $\tilde{V} \cap H^2(\tilde{Q})^3$.

\begin{proposition} \label{propH02}
Let $u_{\varepsilon} \in C^0([0,T_0], \tilde{V})$ be a classical
solution of Problem $(NS_{\varepsilon})$ with initial data
$u_{\varepsilon,0}$ in $H^2(\tilde{Q})^3 \cap \tilde{V}$.  We assume
that $u_{\varepsilon}(t)$ (resp.  $\nabla_h u_{\varepsilon}$) is
uniformly bounded with respect to $\varepsilon$ in
$L^{\infty}((0,T_0], H^{0,1}(\tilde{Q})) \cap
L^2((0,T_0],H^{0,1}(\tilde{Q}))$ (resp.  in
$L^2((0,T_0),H^{0,1}(\tilde{Q}))$).  Then $\partial_3^2
u_{\varepsilon}$ (respectively $\partial_3^2 \nabla_hu_{\varepsilon}$)
is bounded in $L^{\infty}((0,T_0), L^2(\tilde{Q}^3))$ (respectively
$L^{2}((0,T_0), L^2(\tilde{Q}^3))$)
uniformly with respect to $\varepsilon$ and the following estimate
holds, for any $0 \leq t \leq T_0$,
\begin{equation}
\label{estimH02}
\begin{split}
\|\partial_3^2 &u_{\varepsilon}(t) \|_{L^{2}} + \nu_h 
\int_{0}^{t}\|\nabla_h \partial_3^2
u_{\varepsilon}(s)\|_{L^{2}}^{2} ds\cr
& \leq\, \big[\exp (\frac{C}{\nu_h} \int_{0}^{T_0} \|\nabla_h
u_{\varepsilon}(s)\|_{H^{0,1}}^2 ds)\big]
  \Big(\|\partial_3^2 u_{\varepsilon,0}\|_{L^{2}}^{2} \cr
& ~\quad +\frac{C}{\nu_h}
\sup_{0\leq s \leq T_0}\big(\|\partial_3 u_{\varepsilon}(s)\|_{L^2}^2 +  \nu_h
\|\partial_3 u_{\varepsilon}(s)\|_{L^2}
\big)  \int_{0}^{T_0} \|\nabla_h u_{\varepsilon}(s)\|_{H^{0,1}}^2 ds
\Big) .
\end{split}
\end{equation}
\end{proposition}

\begin{proof}
Since $u_{\varepsilon}(t)$ is a very regular solution for $t>0$, all 
the a priori estimates
made below are justified. Differentiating twice the first equation in
$(NS_{\varepsilon})$ with respect to $x_3$ and taking the inner product in
$L^2(\tilde{Q})$ of the resulting equation with $\partial_3^2
u_{\varepsilon}$, we obtain the following equality, for $0 \leq t
\leq T_0$,
\begin{equation*}
\begin{split}
\frac{1}{2} \partial_t \|\partial_3^2 u_{\varepsilon}\|_{L^{2}}^{2}
-\nu_h (\Delta_h \partial_3^2 u_{\varepsilon}, \partial_3^2
u_{\varepsilon}) - \varepsilon (\partial_3^4
u_{\varepsilon},\partial_3^2 u_{\varepsilon}) = -( &\partial_3^2 \nabla
p_{\varepsilon},\partial_3^2 u_{\varepsilon}) \cr
-& (\partial_3^2
(u_{\varepsilon} \cdot \nabla u_{\varepsilon}), \partial_3^2
u_{\varepsilon}) .
\end{split}
\end{equation*}
Since $\partial_3^2 u_{\varepsilon}$ vanishes on $\partial \Omega
\times (-1,1)$ and is periodic in the variable $x_3$, the following
equalities hold:
\begin{equation*}
\begin{split}
-\int_{\tilde{Q}}  \Delta_h \partial_3^2 u_{\varepsilon} \cdot
\partial_3^2 u_{\varepsilon} dx_h dx_3 = &
\int_{\tilde{Q}}|\nabla_h\partial_3^2 u_{\varepsilon}|^2
dx_h dx_3 \cr
- \int_{\tilde{Q}} \partial_3^4 u_{\varepsilon} \cdot \partial_3 ^2
u_{\varepsilon} dx_h dx_3 = & \int_{\tilde{Q}} (\partial_3^3
u_{\varepsilon})^2dx_h dx_3,
\end{split}
\end{equation*}
and
\begin{equation*}
\begin{split}
-\int_{\tilde{Q}} \nabla \partial_3^2 p_{\varepsilon}  \partial_3^2
u_{\varepsilon} dx_h dx_3 = & \int_{\tilde{Q}} \partial_3^2 p_{\varepsilon}
\div \partial_3^2 u_{\varepsilon} dx_h dx_3 -  \int_{\partial
\tilde{Q}}\partial_3^2 p_{\varepsilon} (\partial_3^2 u_{\varepsilon} \cdot
n) d\sigma \cr
=& 0 ~.
\end{split}
\end{equation*}
We deduce from the above equalities that, for $0 \leq t
\leq T_0$,
\begin{equation}
\label{H02aux1}
\begin{split}
\frac{1}{2} \partial_t \|\partial_3^2 u_{\varepsilon}\|_{L^{2}}^{2}
+\nu_h \|\nabla_h \partial_3^2 &u_{\varepsilon}\|_{L^{2}}^{2}
  + \varepsilon \|\partial_3^3 u_{\varepsilon}\|_{L^{2}}^{2} \cr
= &- (\partial_3^2
u_{\varepsilon} \cdot \nabla u_{\varepsilon}, \partial_3^2
u_{\varepsilon})
  - 2 (\partial_3
u_{\varepsilon} \cdot \nabla \partial_3u_{\varepsilon}, \partial_3^2
u_{\varepsilon}) .
\end{split}
\end{equation}
Like in the proof of Lemma \ref{Lunablav}, using the divergence-free
condition (see \eqref{uvv1}), we
decompose the terms in the right hand side of \eqref{H02aux1} as
follows:
\begin{equation}
\label{H02aux2}
\begin{split}
&(\partial_3^2 u_{\varepsilon} \cdot \nabla u_{\varepsilon}, \partial_3^2
u_{\varepsilon}) = (\partial_3^2 u_{\varepsilon,h} \cdot \nabla_h 
u_{\varepsilon}, \partial_3^2
u_{\varepsilon}) -(\partial_3 \div_h u_{\varepsilon,h} \partial_3 
u_{\varepsilon}, \partial_3^2
u_{\varepsilon}) \cr
&(\partial_3 u_{\varepsilon} \cdot \nabla \partial_3u_{\varepsilon}, 
\partial_3^2
u_{\varepsilon}) = (\partial_3 u_{\varepsilon,h} \cdot \nabla_h
\partial_3 u_{\varepsilon}, \partial_3^2 u_{\varepsilon}) -
(\div_h u_{\varepsilon}  \partial_3^2 u_{\varepsilon}, \partial_3^2 
u_{\varepsilon}) .
\end{split}
\end{equation}
Arguing as in the inequality \eqref{uvv3} and applying Lemma
\ref{LestLpLq}, we obtain the estimate
\begin{equation}
\label{H02aux3}
\begin{split}
|(\partial_3^2 u_{\varepsilon,h} \cdot \nabla_h u_{\varepsilon}, \partial_3^2
u_{\varepsilon})| \leq & C \|\partial_3^2
u_{\varepsilon,h}\|_{L^2_v(L^4_h)}^2 \|\nabla_h
u_{\varepsilon}\|_{L^{\infty}_v(L^2_h)}\cr
\leq & C \|\partial_3^2 u_{\varepsilon}\|_{L^{2}} \|\nabla_h
\partial_3^2u_{\varepsilon}\|_{L^{2}} \|\nabla_h
u_{\varepsilon}\|_{H^{0,1}}\cr
\leq&\frac{C}{\nu_h} \|\partial_3^2 u_{\varepsilon}\|_{L^{2}}^2
\|\nabla_h u_{\varepsilon}\|_{H^{0,1}}^2 + \frac{\nu_h}{8}
\|\nabla_h \partial_3^2u_{\varepsilon}\|_{L^{2}}^2 .
\end{split}
\end{equation}
In the same way, we have the estimate
\begin{equation}
\label{H02aux4}
2|(\div_h u_{\varepsilon}  \partial_3^2 u_{\varepsilon}, \partial_3^2 
u_{\varepsilon})| \leq
\frac{C}{\nu_h} \|\partial_3^2 u_{\varepsilon}\|_{L^{2}}^2
\|\nabla_h u_{\varepsilon}\|_{H^{0,1}}^2 + \frac{\nu_h}{8}
\|\nabla_h \partial_3^2u_{\varepsilon}\|_{L^{2}}^2 .
\end{equation}
In order to estimate the term $|(\partial_3 \div_h u_{\varepsilon,h}
\partial_3 u_{\varepsilon}, \partial_3^2 u_{\varepsilon})|$, we
proceed like in \eqref{uvv2}, by applying Lemma \ref{LestLpLq}. We
thus get,
\begin{equation*}
\begin{split}
|(\partial_3 \div_h u_{\varepsilon,h}
\partial_3 u_{\varepsilon}, \partial_3^2 u_{\varepsilon})| \leq
& C \|\partial_3 \nabla_h u_{\varepsilon}\|_{L^{\infty}_v(L^{2}_h)}
  \|\partial_3 u_{\varepsilon}\|_{L^2_v(L^4_h)}
  \|\partial_3^2 u_{\varepsilon}\|_{L^2_v(L^4_h)}\cr
\leq &C \|\partial_3 u_{\varepsilon}\|_{L^{2}}^{1/2} \|\nabla_h
\partial_3 u_{\varepsilon}\|_{L^{2}}^{1/2}
\|\partial_3^2 u_{\varepsilon}\|_{L^{2}}^{1/2} \|\nabla_h
\partial_3^2 u_{\varepsilon}\|_{L^{2}}^{1/2} \cr
& \times \big( \|\partial_3 \nabla_h u_{\varepsilon}\|_{L^{2}} +
\|\partial_3 \nabla_h u_{\varepsilon}\|_{L^{2}}^{1/2}
\|\partial_3^2 \nabla_h u_{\varepsilon}\|_{L^{2}}^{1/2}\big) \cr
\leq &C \|\partial_3 u_{\varepsilon}\|_{L^{2}}^{1/2} \|\nabla_h
\partial_3 u_{\varepsilon}\|_{L^{2}}^{3/2}
\|\partial_3^2 u_{\varepsilon}\|_{L^{2}}^{1/2} \|\nabla_h
\partial_3^2 u_{\varepsilon}\|_{L^{2}}^{1/2} \cr
& +\|\partial_3 u_{\varepsilon}\|_{L^{2}}^{1/2} \|\nabla_h
\partial_3 u_{\varepsilon}\|_{L^{2}}
\|\partial_3^2 u_{\varepsilon}\|_{L^{2}}^{1/2} \|\nabla_h
\partial_3^2 u_{\varepsilon}\|_{L^{2}} .
\end{split}
\end{equation*}
Applying the Young inequalities $2ab \leq a^2 + b^2$ and $ab
\leq (1/4)a^4 +(3/4) b^{4/3}$ to the previous estimates
we obtain,
\begin{equation}
\label{H02aux5}
\begin{split}
|(\partial_3 \div_h u_{\varepsilon,h}
\partial_3 u_{\varepsilon}, &\partial_3^2 u_{\varepsilon})| \cr
\leq &\frac{\nu_h}{8} \|\nabla_h \partial_3^2u_{\varepsilon}\|_{L^{2}}^2 +
\frac{C}{\nu_h} \|\nabla_h
\partial_3 u_{\varepsilon}\|_{L^{2}}^2 \|\partial_3
u_{\varepsilon}\|_{L^{2}}
\|\partial_3^2 u_{\varepsilon}\|_{L^{2}} \cr
&+\frac{C}{\nu_h^{1/3}} \|\nabla_h
\partial_3 u_{\varepsilon}\|_{L^{2}}^2 \|\partial_3
u_{\varepsilon}\|_{L^{2}}^{2/3}
\|\partial_3^2 u_{\varepsilon}\|_{L^{2}}^{2/3}  .
\end{split}
\end{equation}
In the same way, we prove that
\begin{equation}
\label{H02aux6}
\begin{split}
2|(\partial_3 u_{\varepsilon,h} \cdot \nabla_h
\partial_3 u_{\varepsilon}, &\partial_3^2 u_{\varepsilon})| \cr
\leq &
\frac{\nu_h}{8} \|\nabla_h \partial_3^2u_{\varepsilon}\|_{L^{2}}^2 +
\frac{C}{\nu_h} \|\nabla_h
\partial_3 u_{\varepsilon}\|_{L^{2}}^2 \|\partial_3
u_{\varepsilon}\|_{L^{2}}
\|\partial_3^2 u_{\varepsilon}\|_{L^{2}} \cr
&+\frac{C}{\nu_h^{1/3}} \|\nabla_h
\partial_3 u_{\varepsilon}\|_{L^{2}}^2 \|\partial_3
u_{\varepsilon}\|_{L^{2}}^{2/3}
\|\partial_3^2 u_{\varepsilon}\|_{L^{2}}^{2/3}  .
\end{split}
\end{equation}
The equalities \eqref{H02aux1} and \eqref{H02aux2} as well as the
inequalities \eqref{H02aux3} to \eqref{H02aux6} imply that, for $0 \leq t
\leq T_0$,
\begin{equation}
\label{H02aux7}
\begin{split}
\partial_t \|\partial_3^2 u_{\varepsilon}\|_{L^{2}}^{2}
+ &\nu_h \|\nabla_h \partial_3^2 u_{\varepsilon}\|_{L^{2}}^{2}
  + 2 \varepsilon \|\partial_3^3 u_{\varepsilon}\|_{L^{2}}^{2}\cr
  \leq  &\frac{C}{\nu_h}
  \|\nabla_h u_{\varepsilon}\|_{H^{0,1}}^2 \Big(
\|\partial_3^2 u_{\varepsilon}\|_{L^{2}}^2 + \|\partial_3
u_{\varepsilon}\|_{L^{2}}^2 +  \nu_h
\|\partial_3 u_{\varepsilon}\|_{L^{2}} \Big) .
\end{split}
\end{equation}
Integrating the inequality \eqref{H02aux7} from $0$ to $t$, we
obtain, for $0 \leq t \leq T_0$,
\begin{equation*}
\begin{split}
\|\partial_3^2 u_{\varepsilon}(t)&\|_{L^{2}}^{2}
+ \nu_h \int_{0}^{t}\|\nabla_h \partial_3^2
u_{\varepsilon}(s)\|_{L^{2}}^{2} ds
  + 2 \varepsilon\int_{0}^{t}\|\partial_3^3
  u_{\varepsilon}(s)\|_{L^{2}}^{2} ds \cr
  \leq  & \|\partial_3^2 u_{\varepsilon,0}\|_{L^{2}}^{2} +
  \frac{C}{\nu_h} \int_{0}^{t} \|\nabla_h u_{\varepsilon}(s)\|_{H^{0,1}}^2
\|\partial_3^2 u_{\varepsilon}(s)\|_{L^{2}}^2 ds \cr
& +\frac{C}{\nu_h}
\sup_{0\leq s \leq T_0}\big(\|\partial_3 u_{\varepsilon}(s)\|_{L^2}^2 +  \nu_h
\|\partial_3 u_{\varepsilon}(s)\|_{L^2}
\big)  \int_{0}^{T_0} \|\nabla_h u_{\varepsilon}(s)\|_{H^{0,1}}^2 ds  .
\end{split}
\end{equation*}
Applying the Gronwall lemma, we deduce from the previous inequality
that, for $0 \leq t \leq T_0$,
\begin{equation}
\label{H02aux8}
\begin{split}
\|\partial_3^2 &u_{\varepsilon}(t) \|_{L^{2}}^{2}
+ \nu_h \int_{0}^{t}\|\nabla_h \partial_3^2
u_{\varepsilon}(s)\|_{L^{2}}^{2} ds
  + 2 \varepsilon\int_{0}^{t}\|\partial_3^3
  u_{\varepsilon}(s)\|_{L^{2}}^{2} ds \cr
& \leq  \, \big[\exp (\frac{C}{\nu_h} \int_{0}^{T_0} \|\nabla_h
u_{\varepsilon}(s)\|_{H^{0,1}}^2 ds)\big]
  \Big(\|\partial_3^2 u_{\varepsilon,0}\|_{L^{2}}^{2} \cr
& \quad \quad +\frac{C}{\nu_h}
\sup_{0\leq s \leq T_0}\big(\|\partial_3 u_{\varepsilon}(s)\|_{L^2}^2 +  \nu_h
\|\partial_3 u_{\varepsilon}(s)\|_{L^2}
\big)  \int_{0}^{T_0} \|\nabla_h u_{\varepsilon}(s)\|_{H^{0,1}}^2 ds
\Big) .
\end{split}
\end{equation}
The proposition is thus proved.
\end{proof}

The propositions \ref{Cauchy} and \ref{propH02} together with Remark
\ref{H02} imply the following $H^{0,2}$-propagation result.

\begin{corollary} \label{propageH02}
Let $u_0 \in \widetilde{H}^{0,1}(Q) \cap \widetilde{H}^{0,2}_0(Q)$ be 
given.  Let
$\varepsilon_m >0 $ be a
(decreasing) sequence converging to $zero$ and
$u_0^m \in {\tilde H}^1_0(Q) \cap H^2(Q) \cap 
\widetilde{H}^{0,2}_0(Q)$ be a sequence of initial data converging
to $u_0$ in $\widetilde{H}^{0,2}(Q)$, when $m$ goes to infinity. 
Assume that the
system $(NS_{\varepsilon_m})$, with initial data $\Sigma u_0^m$, has 
a strong solution
$u_{\varepsilon_m}(t) \in C^0((0,T_0),\widetilde{V})$ where $T_0$ does
not depend on $\varepsilon_m$ and that the sequences $u_{\varepsilon_m}(t)$
and $\nabla_h u_{\varepsilon_m}(t)$ are uniformly bounded in
$L^{\infty}((0, T_0),H^{0,1}(\tilde{Q}))$ and
$L^{2}((0, T_0),H^{0,1}(\tilde{Q}))$ respectively. Then,
the sequence $u_{\varepsilon_m}(t)$ converges
in $L^{\infty}((0,T_0), L^2(\tilde{Q})^3) \cap L^2((0, T_0),
H^{1,0}(\tilde{Q}))$ to a solution $u^*\in L^{\infty}((0,T_0),$ $
H^{0,2}(\tilde{Q}))$ of the problem $(NS_h)$, such that
  $\nabla_h \partial_3^i u^*$ belongs to
$L^{2}((0, T_0), L^2(\tilde{Q})^3)$, for $i=0,1,2$. Moreover, the
solution
$u^*$ belongs to $L^{\infty}((0,T_0),\widetilde{H}^{0,2}_0(Q))$.

\end{corollary}

\begin{proof} Let $u_0 \in \widetilde{H}^{0,1}(Q) \cap 
\widetilde{H}^{0,2}_0(Q)$ be given.  We
notice that, by Remark \ref{H02}, there exists a sequence $u_0^m \in
{\tilde H}^1_0(Q) \cap H^2(Q) \cap \widetilde{H}^{0,2}_0(Q)$ of
initial data converging to $u_0$ in $\widetilde{H}^{0,2}(Q)$, when $m$ goes to
infinity. Let $u_0^m$ be such a sequence. As we have remarked in the
introduction, $\Sigma u_0^m$ belongs to $H^2(\tilde{Q})^3$ and
$\partial_3 u_{0,h}^m$ vanishes on $\Gamma_0 \cup \Gamma_1$. By Proposition
\ref{propH02}, the classical solution $u_{\varepsilon_m}$ of
$(NS_{\varepsilon_m})$ is more regular in the sense that
$\partial_3^2 u_{\varepsilon_m}$ (respectively
$\nabla_h \partial_3^2 u_{\varepsilon_m}$)
is uniformly bounded in $L^{\infty}((0, T_0),
L^2(\tilde{Q})^3)$ (respectively in $L^{2}((0, T_0),
L^2(\tilde{Q})^3)$. Thus the limit $\partial_3^2  u^*$ belongs to
$L^{\infty}((0, T_0), L^2(\tilde{Q})^3)$ and
$\nabla_h \partial_3^2 u^*$ belongs to
$L^{2}((0, T_0), L^2(\tilde{Q})^3)$.
\end{proof}

\section{Global existence results for small initial data}

  We begin with the simplest result.

  \begin{theorem} \label{Global1}
  There exists a positive constant $c_0$ such that, if $u_0$ belongs
  to $\widetilde{H}^{0,1}(Q)$ and $\|u_0\|_{H^{0,1}}\leq
c_0\nu_h$, then the system $(NS_h)$ admits a (unique) global solution
$u(t)$, with $u(0)=u_0$, such that
$$
   u \in L^\infty(\R_+,\widetilde{H}^{0,1}(Q)) ~\quad \hbox{ and }\quad
\partial_3 \nabla_h u \in  L^2(\R_+, L^2(Q)^3).
$$
  \end{theorem}

  \begin{proof}
According to the strategy explained in the introduction and according to
Proposition \ref{Cauchy}, it is sufficient to prove that
there exists a positive constant $c_1$ such that if
$u_{\varepsilon}(0) =w_0$ belongs to
  $H^1_{0,per}(\tilde{Q})$ and satisfies
  $$
\|w_0\|_{H^{0,1}}\leq c_1\nu_h,
$$
then, for any $\varepsilon>0$, the equations $(NS_{\varepsilon})$ 
admit a unique global
solution $u_{\varepsilon}(t) \in C^0([0, +\infty), \tilde{V})$ with
$u_{\varepsilon}(0)=w_0$ and moreover,  $u_{\varepsilon}$ and
$\partial_3 \nabla_h u_{\varepsilon}$ are uniformly bounded (with respect to
$\varepsilon$) in $L^{\infty}((0,+\infty), H^{0,1}(\tilde{Q})$ and
$L^2((0, +\infty), L^2(\tilde{Q})^3)$.
\vskip 1mm
Let now $u_{\varepsilon}$ be the local solution of the equations
$(NS_{\varepsilon})$ with $u_{\varepsilon}(0)=w_0$. Since
$u_{\varepsilon}$ is a classical solution on the maximal interval of
existence, all the a priori estimates made below can be justified
rigorously. Differentiating the first equation in $(NS_{\varepsilon})$ with
respect to $x_3$ and taking the inner product in $L^2(\tilde{Q})^3$
of it with $\partial_3 u_{\varepsilon}$, we obtain, for $0 \leq t
\leq T_\varepsilon$, where $T_\varepsilon >0$ is the maximal time of
existence,
\begin{equation}
\label{equD3}
\begin{split}
\frac{1}{2} \partial_t \|\partial_3 u_{\varepsilon}\|_{L^2}^{2} - 
\nu_h (\Delta_h
\partial_3 u_{\varepsilon} &,\partial_3 u_{\varepsilon})
- \varepsilon (\partial_3^3 u_{\varepsilon},\partial_3
u_{\varepsilon}) \cr
& = -(\nabla \partial_3 p_{\varepsilon}, \partial_3
u_{\varepsilon}) -(\partial_3(u_{\varepsilon} \nabla u_{\varepsilon}),
\partial_3 u_{\varepsilon}).
\end{split}
\end{equation}
Since $u_{\varepsilon}$ and hence $\partial_3 u_{\varepsilon}$ vanish 
on the lateral boundary $\partial
\Omega \times (-1,1)$ and that $u_{\varepsilon}$ and
$p_{\varepsilon}$ are periodic in the
vertical variable, we have,
\begin{equation}
\label{globaux1}
\begin{split}
-\int_{\tilde{Q}}  \Delta_h \partial_3 u_{\varepsilon} \cdot
\partial_3 u_{\varepsilon} dx_h dx_3 = & 
\int_{\tilde{Q}}|\nabla_h\partial_3 u_{\varepsilon}|^2
dx_h dx_3 \cr
- \int_{\tilde{Q}} \partial_3^3 u_{\varepsilon} \cdot \partial_3
u_{\varepsilon} dx_h dx_3 = & \int_{\tilde{Q}} (\partial_3^2
u_{\varepsilon})^2dx_h dx_3,
\end{split}
\end{equation}
and
\begin{equation}
\label{globaux2}
\begin{split}
-\int_{\tilde{Q}} \nabla \partial_3 p_{\varepsilon}  \partial_3
u_{\varepsilon} dx_h dx_3 = & \int_{\tilde{Q}} \partial_3 p_{\varepsilon}
\div \partial_3 u_{\varepsilon} dx_h dx_3 -  \int_{\partial
\tilde{Q}}\partial_3 p_{\varepsilon} (\partial_3 u_{\varepsilon} \cdot
n) d\sigma \cr
=& 0 ~.
\end{split}
\end{equation}
The equalities \eqref{equD3}, \eqref{globaux1}, and \eqref{globaux2}
together with Lemma \ref{Lunablav} imply that, for $0 \leq t \leq
T_\varepsilon$,
\begin{equation}
\label{energie2aux}
\partial_t \|\partial_3 u_{\varepsilon}\|_{L^2}^{2}
+ 2 \nu_h \|\nabla_h\partial_3 u_{\varepsilon}|\|_{L^2}^{2}
+ 2\varepsilon \| \partial_3^2 u_\varepsilon\|_{L^2}^{2} \leq  4 C_1
\|u_{\varepsilon}\|_{H^{0,1}} \| \nabla_h
u_{\varepsilon}\|_{H^{0,1}}^2
\end{equation}
We deduce from the estimates \eqref{enaux1} and \eqref{energie2aux}
that, for
$0 \leq t \leq T_\varepsilon$,
\begin{equation}
\label{energie3aux}
\begin{split}
\partial_t \big(\|u_{\varepsilon}\|_{L^2}^{2}  +\|\partial_3
u_{\varepsilon}\|_{L^2}^{2} \big) + 2 \nu_h \big( \|\nabla_h 
u_{\varepsilon}|\|_{L^2}^{2}
&+ \|\nabla_h\partial_3 u_{\varepsilon}|\|_{L^2}^{2}\big)  \cr
&+ 2 \varepsilon
\big(\| \partial_3 u_\varepsilon\|_{L^2}^{2} + \| \partial_3^2
u_\varepsilon\|_{L^2}^{2} \big) \cr
\leq 8 C_1 \|u_{\varepsilon}\|_{H^{0,1}}& \big(  \|\nabla_h 
u_{\varepsilon}|\|_{L^2}^{2}
+ \|\nabla_h\partial_3 u_{\varepsilon}|\|_{L^2}^{2}\big).
\end{split}
\end{equation}
Suppose now that the initial data $u_{\varepsilon}(0)=w_0$ are small enough
in the sense that
\begin{equation}
\label{w0petit}
\|w_0\|_{H^{0,1}} \leq \frac{\nu_h}{32 C_1}
\end{equation}
Then, by continuity, there exists a time interval
$[0,\tau_{\varepsilon})$ such that, for $t \in
[0,\tau_{\varepsilon})$, $\|u_{\varepsilon}(t)\|_{H^{0,1}} < \nu_h
/ (8 C_1)$. If $\tau_{\varepsilon} < T_{\varepsilon}$, then
$\|u_{\varepsilon}(\tau_{\varepsilon})\|_{H^{0,1}} = \nu_h
/ (8 C_1)$. Assume now that $\tau_{\varepsilon} < T_{\varepsilon}$.
If $t$ belongs to the time interval $[0,\tau_{\varepsilon}]$, we
deduce from the inequality \eqref{energie3aux} that
\begin{equation}
\label{energie4aux}
\begin{split}
\partial_t \big(\|u_{\varepsilon}\|_{L^2}^{2}  +\|\partial_3
u_{\varepsilon}\|_{L^2}^{2} \big) + \nu_h \big( \|\nabla_h 
u_{\varepsilon}\|_{L^2}^{2}
&+ \|\nabla_h\partial_3 u_{\varepsilon}\|_{L^2}^{2}\big)  \cr
&+ 2 \varepsilon
\big(\| \partial_3 u_\varepsilon\|_{L^2}^{2} + \| \partial_3^2
u_\varepsilon\|_{L^2}^{2} \big)  \leq 0.
\end{split}
\end{equation}
Integrating the inequality \eqref{energie4aux} from $0$ to $t$,
we obtain that, for  $t \leq \tau_{\varepsilon}$,
\begin{equation}
\label{energie5aux}
\begin{split}
\|u_{\varepsilon}(t)\|_{L^2}^{2}  + &\|\partial_3
u_{\varepsilon}(t)\|_{L^2}^{2} +\nu_h \int_{0}^{t}
\big( \|\nabla_h u_{\varepsilon}(s)\|_{L^2}^{2}
+ \|\nabla_h\partial_3 u_{\varepsilon}(s)\|_{L^2}^{2}\big) ds \cr
&+ 2 \varepsilon \int_{0}^{t}
\big(\| \partial_3 u_\varepsilon (s)\|_{L^2}^{2} + \| \partial_3^2
u_\varepsilon (s)\|_{L^2}^{2} \big) ds \leq  \|w_0\|_{L^2}^{2}  +\|\partial_3
w_0\|_{L^2}^{2}
\end{split}
\end{equation}
The estimate \eqref{energie5aux} implies that, for $t \leq \tau_{\varepsilon}$,
$$
\|u_{\varepsilon}(t)\|_{H^{0,1}} \leq \frac{\nu_h}{16 C_1}.
$$
In particular, $\|u_{\varepsilon}(\tau_{\varepsilon})\|_{H^{0,1}}
\leq \nu_h / (16 C_1)$, which contradicts the definition of
$\tau_{\varepsilon}$. Thus $\tau_{\varepsilon} = T_{\varepsilon}$ and
one deduces from \eqref{energie5aux} that, for $0 \leq t
<T_{\varepsilon}$,
\begin{equation}
\label{energie5auxbis}
\begin{split}
\|u_{\varepsilon}(t)\|_{L^2}^{2}  + &\|\partial_3
u_{\varepsilon}(t)\|_{L^2}^{2} +\nu_h \int_{0}^{T_{\varepsilon}}
\big( \|\nabla_h u_{\varepsilon}(s)\|_{L^2}^{2}
+ \|\nabla_h\partial_3 u_{\varepsilon}(s)\|_{L^2}^{2}\big) ds \cr
&+ 2 \varepsilon \int_{0}^{T_{\varepsilon}}
\big(\| \partial_3 u_\varepsilon (s)\|_{L^2}^{2} + \| \partial_3^2
u_\varepsilon (s)\|_{L^2}^{2} \big) ds \leq  \|w_0\|_{L^2}^{2}  +\|\partial_3
w_0\|_{L^2}^{2}
\end{split}
\end{equation}
\vskip 1mm
To prove that $u_{\varepsilon}(t)$ exists globally, that is, that
$T_{\varepsilon}=+ \infty$, it remains to show that $\|\nabla_h
u_{\varepsilon}(t)\|_{L^2}$ is uniformly bounded with respect to $t
\in [0, T_{\varepsilon})$. But this property is a direct consequence
of Proposition \ref{gradh}. Theorem \ref{Global1} is thus proved.
\end{proof}
\vskip 3mm

A more careful analysis allows to prove the following global
existence result.

  \begin{theorem} \label{Global2}
  There exist positive constants $c_0$ and $c_0^*$ such that, if $u_0$ belongs
  to $\widetilde{H}^{0,1}(Q)$ and satisfies the following smallness condition
  $$
  \|\partial_3 u_0\|_{L^2(\Omega)}^{\frac
12}\|u_0\|_{L^2(\Omega)}^{\frac
12}\exp(\frac{c_0\|u_0\|_{L^2}^2}{\nu_h^2})\leq c_0^*\nu_h,
$$
  then the system $(NS_h)$ admits a (unique) global solution
$u(t)$, with $u(0)=u_0$, such that
$$
   u \in L^\infty(\R_+,\widetilde{H}^{0,1}(Q)) ~\quad \hbox{ and }\quad
\partial_3\nabla_h u \in  L^2(\R_+; L^2(Q)^3).
$$
  \end{theorem}

  \begin{proof}
  Like in the proof of Theorem \ref{Global1},
  it is sufficient to prove that
there exist positive constants $c_1$ and $c_1^*$ such that if
$u_{\varepsilon}(0) =w_0$ belongs to
  $H^1_{0,per}(\tilde{Q})$ and satisfies
  $$
\|\partial_3 w_0\|_{L^2(\Omega)}^{\frac
12}\|w_0\|_{L^2(\Omega)}^{\frac
12}\exp(c_1\frac{\|w_0\|_{L^2}^2}{\nu_h^2})\leq c_1^*\nu_h,
$$
then, for any $\varepsilon>0$, the equations $(NS_{\varepsilon})$ 
admit a unique global
solution $u_{\varepsilon}(t) \in C^0([0, +\infty), \tilde{V})$ with
$u_{\varepsilon}(0)=w_0$ and moreover,  $u_{\varepsilon}$ and
$\nabla_h u_{\varepsilon}$ are uniformly bounded (with respect to
$\varepsilon$) in $L^{\infty}((0,+\infty), H^{0,1}(\tilde{Q})$ and
$L^2((0, +\infty), H^{0,1}(\tilde{Q})$.
\vskip 1mm
Let now $u_{\varepsilon}$ be the local solution of the equations
$(NS_{\varepsilon})$ with $u_{\varepsilon}(0)=w_0$ and let
$T_{\varepsilon}>0$ be the maximal time of existence. Like in the
proof of Theorem \ref{Global1}, $u_{\varepsilon}$ satisfies the
equality \eqref{equD3}. But here, in order to estimate the term
$(\partial_3(u_{\varepsilon}\nabla u_{\varepsilon}), \partial_3
u_{\varepsilon})$, we take into account the estimates \eqref{uvv2} and
\eqref{uvv3}, instead of directly applying Lemma \ref{Lunablav}.
  The equalities  \eqref{equD3}, \eqref{globaux1}, \eqref{globaux2},
\eqref{uvv1} and, the estimates \eqref{uvv2} and \eqref{uvv3} imply 
that, for $0 \leq
t < T_{\varepsilon}$,
\begin{equation}
\label{Glob2aux1}
\begin{split}
\partial_t \|\partial_3 u_{\varepsilon}&\|_{L^2}^{2}
+ 2 \nu_h \|\nabla_h\partial_3 u_{\varepsilon}|\|_{L^2}^{2}
+ 2\varepsilon \| \partial_3^2 u_\varepsilon\|_{L^2}^{2} \cr
  \leq  C_2&\big(\|\nabla_h u_{\varepsilon}\|_{L^2}^{1/2}\|\partial_3 
u_{\varepsilon}\|_{L^2}
  \|\nabla_h \partial_3 u_{\varepsilon}\|_{L^2}^{3/2}+\|\nabla_h 
u_{\varepsilon}\|_{L^2}
  \|\partial_3 u_{\varepsilon}\|_{L^2}
\|\nabla_h \partial_3 u_{\varepsilon}\|_{L^2} \big) .
\end{split}
\end{equation}
Using the Young inequalities $ab \leq \frac{3}{4} a^{\frac
43}+\frac{1}{4} b^4$ and  $ab \leq \frac{1}{2}a^2+ \frac{1}{2}b^2$,
we get the following estimates,
\begin{equation}
\label{Glob2aux2}
\begin{split}
C_2 \|\nabla_h u_{\varepsilon}\|_{L^2}^{1/2}\|\partial_3
u_{\varepsilon}\|_{L^2} &\|\nabla_h
\partial_3 u_{\varepsilon}\|_{L^2}^{3/2} \cr
&\leq \frac {27C_2^4}{32\nu_h^3}\|\nabla_h
u_{\varepsilon}\|_{L^2}^2\|\partial_3
u_{\varepsilon}\|_{L^2}^4+\frac{\nu_h}{2}\|\nabla_h\partial_3 
u_{\varepsilon}\|_{L^2}^2
\end{split}
\end{equation}
and
\begin{equation}
\label{Glob2aux3}
\begin{split}
C_2 \|\nabla_h u_{\varepsilon}\|_{L^2}\|\partial_3
u_{\varepsilon}\|_{L^2}&\|\nabla_h \partial_3
u_{\varepsilon}\|_{L^2} \cr
& \leq \frac{C_2^2}{2\nu_h}\|\nabla_h u_{\varepsilon}\|_{L^2}^2\|\partial_3
u_{\varepsilon}\|_{L^2}^2+\frac{\nu_h}{2}\|\nabla_h\partial_3 
u_{\varepsilon}\|_{L^2}^2 .
\end{split}
\end{equation}
{}From the estimates \eqref{Glob2aux1}, \eqref{Glob2aux2} and
\eqref{Glob2aux3}, we deduce that,  for $0 \leq
t < T_{\varepsilon}$,
\begin{equation}
\label{Glob2aux4}
\begin{split}
\frac{d}{dt}\|\partial_3 u_{\varepsilon}(t)\|_{L^2}^2 &+ \nu_h 
\|\nabla_h \partial_3
u_{\varepsilon}(t)\|_{L^2}^2  + 2\varepsilon \| \partial_3^2
u_\varepsilon\|_{L^2}^{2} \cr
  & \leq C_3 \Big( \frac{1}{\nu_h} \|\nabla_h 
u_{\varepsilon}\|_{L^2}^2\|\partial_3
u_{\varepsilon}\|_{L^2}^2+ \frac{1}{\nu_h^3} \|\nabla_h 
u_{\varepsilon}\|_{L^2}^2
\|\partial_3 u_{\varepsilon}\|_{L^2}^4 \Big),
\end{split}
\end{equation}
where $C_3 = \max( C_2^2/2, 27C_2^4 /32)$. The inequality
\eqref{Glob2aux4} shows that, if there exists $\tau_{\varepsilon} 
<T_{\varepsilon}$
such that $\|\partial_3 
u_{\varepsilon}(\tau_{\varepsilon})\|_{L^2}^2$ vanishes, then
$\|\partial_3 u_{\varepsilon}(t)\|_{L^2}^2$ is identically equal to
zero for $\tau_{\varepsilon} \leq t <T_{\varepsilon}$.\HB
On the time interval $[0,\tau_{\varepsilon})$, the inequality
\begin{equation}
\label{Glob2aux4Bis}
\begin{split}
\frac{d}{dt}\|\partial_3 u_{\varepsilon}(t)\|_{L^2}^{2}\leq
C_3 \Big(
\frac{1}{\nu_h}\|\nabla_h u_{\varepsilon}\|_{L^2}^2& \|\partial_3
u_{\varepsilon}\|_{L^2}^2  \cr
&+\frac{1}{\nu_h ^3}\|\nabla_h u_{\varepsilon}\|_{L^2}^2
\|\partial_3 u_{\varepsilon}\|_{L^2}^4\Big),
\end{split}
\end{equation}
can be written as
$$
-\frac{d}{dt}  \|\partial_3 u_{\varepsilon}(t)\|_{L^2}^{-2}
  \leq
\frac{C_3}{\nu_h}\|\nabla_h u_{\varepsilon}\|_{L^2}^2 \|\partial_3
u_{\varepsilon}\|_{L^2}^{-2} +\frac{C_3}{\nu_h^3}\|\nabla_h
u_{\varepsilon}\|_{L^2}^2 ,
$$
or also
\begin{equation*}
\begin{split}
-\frac{d}{dt}\Big(\|\partial_3
u_{\varepsilon}\|_{L^2}^{-2}\exp(\frac{C_3}{\nu_h}\int_0^t &\|\nabla_h
u_{\varepsilon}(s)\|_{L^2}^2 ds)\Big) \cr
&\leq \frac{C_3}{\nu_h^3}\|\nabla_h
u_{\varepsilon}(t)\|_{L^2}^2\exp(\frac{C_3}{\nu_h}\int_0^t\|\nabla_h
u_{\varepsilon}(s)\|_{L^2}^2ds) .
\end{split}
\end{equation*}
Integrating this inequality from $0$ to $t$,
we obtain, for $0 \leq t <\tau_\varepsilon$,
\begin{equation}
\label{Glob2aux5}
\begin{split}
\|\partial_3 w_0\|_{L^2}^{-2}-\|\partial_3
u_{\varepsilon}&(t)\|_{L^2}^{-2} \exp(\frac{C_3}{\nu_h}\int_0^t \|\nabla_h
u_{\varepsilon}(s)\|_{L^2}^2 ds) \cr
& \leq \frac{C_3}{\nu_h^3}\int_0^t \|\nabla_h
u_{\varepsilon}(s)\|^2_{L^2}ds \times
\exp(\frac{C_3}{\nu_h}\int_0^t \|\nabla_h
u_{\varepsilon}(s)\|_{L^2}^2ds) .
\end{split}
\end{equation}
The second energy estimate in Lemma \ref{energie} and the inequality
\eqref{Glob2aux5} imply that, for $0 \leq t <\tau_\varepsilon$,
\begin{equation}
\label{Glob2aux6}
\begin{split}
\|\partial_3 w_0\|_{L^2}^{-2}- \frac{C_3}{\nu_h^4}
\|w_0\|_{L^2}^2 \exp(\frac{C_3\|w_0\|_{L^2}^2}{\nu_h^2})
\leq
\|\partial_3 u_{\varepsilon}(t)\|_{L^2}^{-2}\exp(\frac{C_3
\|w_0\|_{L^2}^2}{\nu_h^2}) .
\end{split}
\end{equation}
Thus, if we assume that,
$$
\|\partial_3
w_0\|_{L^2}^{-2}- \frac{C_3 \|w_0\|_{L^2}^2}{\nu_h^4}
\exp(\frac{C_3\|w_0\|_{L^2}^2}{\nu_h^2})>0,
$$
that is,
\begin{equation}
\label{Glob2aux7}
\begin{split}
\|\partial_3 w_0\|_{L^2}^{1/2}\|w_0\|_{L^2}^{1/2}
\exp(\frac{C_3\|w_0\|_{L^2}^2}{4 \nu_h^2}) < C_3^{-1/4}\nu_h,
\end{split}
\end{equation}
then, we get the following uniform bound, for $0 \leq t <\tau_\varepsilon$,
\begin{equation}
\label{Glob2aux7bis}
\begin{split}
\|\partial_3 u_{\varepsilon}(t)\|_{L^2}^2\leq
\exp(\frac{C_3\|w_0\|_{L^2}^2}{\nu_h^2})
\Big(\|\partial_3 w_0\|_{L^2}^{-2} -\frac{C_3}{\nu_h^4}
\|w_0\|_{L^2}^2\exp(\frac{C_3\|w_0\|_{L^2}^2}{\nu_h^2})\Big)^{-1} .
\end{split}
\end{equation}
Let us denote $B_0$ the right-hand side term of the inequality
\eqref{Glob2aux7bis}. Integrating the estimate \eqref{Glob2aux4} from
$0$ to $t$ and taking into account the second energy estimate in
Lemma \ref{energie} as well as the estimate \eqref{Glob2aux7bis} and the
definition of $\tau_{\varepsilon}$, we at once obtain the following
inequality, for any $0 \leq t <T_{\varepsilon}$,
\begin{equation}
\label{Glob2aux8}
\begin{split}
\nu_h  \int_{0}^{t}  \|\nabla_h \partial_3
u_{\varepsilon}(s)\|_{L^2}^2 ds + 2\varepsilon \int_{0}^{t} &\| \partial_3^2
u_\varepsilon(s)\|_{L^2}^{2} ds  \cr
& \leq
\frac{C_3}{\nu_h} B_0^2 \big(1 + \frac{1}{\nu_h^2} B_0^2 \big)
\int_{0}^{t} \|\nabla_h u_{\varepsilon}(s)\|_{L^2}^2 ds  \cr
& \leq \frac{C_3}{2 \nu_h^2} B_0^2 \big(1 + \frac{1}{\nu_h^2} B_0^2
\big) \|w_0\|_{L^{2}}^{2}
\end{split}
\end{equation}

\noindent To prove that $u_{\varepsilon}(t)$ exists globally, that is, that
$T_{\varepsilon}=+ \infty$, it remains to show that $\|\nabla_h
u_{\varepsilon}(t)\|_{L^2}$ is uniformly bounded with respect to $t
\in [0, T_{\varepsilon})$.  Like in the proof of Theorem \ref{Global1},
this property is a direct consequence of Proposition \ref{gradh}.
Theorem \ref{Global2} is thus proved.
\end{proof}

\begin{remark} \label{ApplicThGlob2}
The previous theorem allows to take large initial data in the
following sense. For example,  we can take $u_0 \in
\widetilde{H}^{0,1}(Q)$ such that,
$$
\|u_0\|_{L^2(Q)}\leq C\eta^{\alpha}
$$
and
$$
\|\partial_3 u_0\|_{L^2(Q)}\leq C\eta^{-\alpha},
$$
where $\eta$ is a small positive constant going to $0$ and $C>0$ is an
appropriate positive constant
\end{remark}

\begin{remark} \label{IntegrerBisThGlob2}
Let us come back to the inequality \eqref{Glob2aux4Bis}. If we set
$$
y(t)= \frac{\nu_h}{2} + \frac{1}{\nu_h}\|\partial_3
u_{\varepsilon}(t)\|_{L^2}^2~, \quad g(t) =\frac{C_3}{\nu_h}
\|\nabla_h u_{\varepsilon}(t)\|_{L^2}^2~,
$$
the inequality \eqref{Glob2aux4Bis} becomes, for $0 \leq t <
\tau_{\varepsilon}$,
\begin{equation*}
\frac{dy}{dt}(t) \leq  g(t) \, y^2(t).
\end{equation*}
Integrating this inequality from $0$ to $t$, for $0\leq t
<\tau_{\varepsilon}$, we get
$$
- \frac{1}{y(t)} + \frac{1}{y(0)} \leq \int_{0}^{t} g(s) ds,
$$
or also,
$$
y(t) \leq \frac{y(0)}{1 -y(0) \int_{0}^{t} g(s) ds},
$$
as long as $1 -y(0) \int_{0}^{t} g(s) ds >0$.
The previous estimate also writes
\begin{equation}
\label{Glob2Rem1}
\begin{split}
\frac{\nu_h}{2} + \frac{1}{\nu_h}\|\partial_3
u_{\varepsilon}(t)\|_{L^2}^2 \leq \,  &\Big(\frac{\nu_h}{2} + 
\frac{1}{\nu_h}\|\partial_3
u_{\varepsilon}(0)\|_{L^2}^2 \Big) \cr
  \times &\Big(1 -  \frac{C_3}{\nu_h} \int_{0}^{t}
\|\nabla_h u_{\varepsilon}(s)\|_{L^2}^2 ds
\big(\frac{\nu_h}{2} + \frac{1}{\nu_h}\|\partial_3
u_{\varepsilon}(0)\|_{L^2}^2 \big)\Big)^{-1}
\end{split}
\end{equation}
Inequality \eqref{Glob2Rem1} and Lemma \ref{energie} imply that
\begin{equation}
\label{Glob2Rem2}
\begin{split}
\frac{\nu_h}{2} + \frac{1}{\nu_h}\|\partial_3
u_{\varepsilon}(t)\|_{L^2}^2 \leq \,  &\Big(\frac{\nu_h}{2} + 
\frac{1}{\nu_h}\|\partial_3
u_{\varepsilon}(0)\|_{L^2}^2 \Big) \cr
  \times &\Big(1 -  \frac{C_3}{2\nu_h^2}
\|u_{\varepsilon}(0)\|_{L^2}^2
\big(\frac{\nu_h}{2} + \frac{1}{\nu_h}\|\partial_3
u_{\varepsilon}(0)\|_{L^2}^2 \big)\Big)^{-1}
\end{split}
\end{equation}
Thus, if
\begin{equation}
\label{Glob2Rem3}
C_3\|u_{\varepsilon}(0)\|_{L^2}^2
\big(\frac{\nu_h}{2} + \frac{1}{\nu_h}\|\partial_3
u_{\varepsilon}(0)\|_{L^2}^2 \big) \leq \nu_h^2,
\end{equation}
we obtain the following uniform bound, for $0 \leq t <
\tau_{\varepsilon}$,
\begin{equation}
\label{Glob2Rem4}
\frac{\nu_h}{2} + \frac{1}{\nu_h}\|\partial_3
u_{\varepsilon}(t)\|_{L^2}^2 \leq \,  2\Big(\frac{\nu_h}{2} + 
\frac{1}{\nu_h}\|\partial_3
u_{\varepsilon}(0)\|_{L^2}^2 \Big) .
\end{equation}
Like in the proof of Theorem \ref{Global2}, we deduce that, under
the condition \eqref{Glob2Rem3}, the solution $u_{\varepsilon}(t)$
exists globally.

\end{remark}

\vskip 3mm

Theorems \ref{Global1} and \ref{Global2} together with Corollary
\ref{propageH02} at once imply the following result of propagation of
regularity.

\begin{corollary} \label{propageH02Bis}
Under the hypotheses of Theorem \ref{Global1} or \ref{Global2}, if
moreover the initial data $u_0$ belong to $\widetilde{H}^{0,2}_0(Q)$, then  the
solution $u$ of System $(NS_h)$ with $u(0)=u_0$ belongs to
$L^{\infty}(\R_+,\widetilde{H}^{0,2}_0(Q))$ and
$\partial_3^2 \nabla_h u$ belongs to $L^2(\R_+, L^2(Q)^3)$.
  \end{corollary}


\section{The case of general initial data}

In this section, we want to prove the local existence of the solution $u(t)$
of the equations $(NS_h)$, when the initial data are not necessarily
small.

  \begin{theorem} \label{Local}
  Let $U_0$ be given in  $\widetilde{H}^{0,1}(Q)$. There exist a positive time
  $T_0$ and a positive constant $\eta$ such that, if $u_0$ belongs to
$\widetilde{H}^{0,1}(Q)$ and $\|U_0 -u_0\|_{H^{0,1}} \leq \eta$, then
the system $(NS_h)$ admits a
(unique) strong solution $u(t)$, with $u(0)=u_0$, such that
$$
   u \in L^\infty((0,T_0),\widetilde{H}^{0,1}(Q)) ~\quad \hbox{ and }\quad
\partial_3 \nabla_h u \in  L^2((0,T_0), L^2(Q)^3).
$$
  \end{theorem}

\begin{proof}
According to the strategy explained in the introduction and according to
Proposition \ref{Cauchy}, it is sufficient to prove that there exist
positive constants $\eta$ and $T_0$ such that, if
$u_{\varepsilon}(0) =v_0$ belongs to
  $H^1_{0,per}(\tilde{Q})$ and satisfies
\begin{equation}
\label{v0U0}
  \|v_0 - \Sigma U_0\|_{H^{0,1}(\tilde{Q})}\leq \eta,
\end{equation}
then, for any $\varepsilon>0$ small enough, the equations
$(NS_{\varepsilon})$ admit a unique (local)
solution $u_{\varepsilon}(t) \in C^0([0, T_0], \tilde{V})$ with
$u_{\varepsilon}(0)=v_0$ and moreover,  $u_{\varepsilon}$ and
$\partial_3 \nabla_h u_{\varepsilon}$ are uniformly bounded (with respect to
$\varepsilon$) in $L^{\infty}((0,T_0), H^{0,1}(\tilde{Q}))$ and
$L^2((0, T_0), L^2(\tilde{Q})^3)$. \HB
Let $u_{\varepsilon}(t)$ be the strong solution of the equations
$(NS_{\varepsilon})$ with initial data $u_{\varepsilon}(0)=v_0 \in
H^1_{0,per}(\tilde{Q})$ satisfying the condition \eqref{v0U0}. Let
$T_{\varepsilon}>0$ be the maximal time of existence of this solution.
The proof of Theorem \ref{Global2} and Remark
\ref{IntegrerBisThGlob2} show that, if
$$
\frac{C_3}{\nu_h} \int_{0}^{\tau}
\|\nabla_h u_{\varepsilon}(s)\|_{L^2}^2 ds
\big(\frac{\nu_h}{2} + \frac{1}{\nu_h}\|\partial_3
u_{\varepsilon}(0)\|_{L^2}^2 \big) < 1,
$$
then $T_{\varepsilon} > \tau$.  \HB
It is thus sufficient to show that, for $\eta >0$ small enough, there
exist a positive constant $T_0$ such that, for any $\varepsilon >0$, the strong
solution $u_{\varepsilon}$ of $(NS_{\varepsilon})$ satisfies the
inequality
\begin{equation}
\label{bornexist0}
\frac{C_3}{\nu_h} \int_{0}^{T_0}
\|\nabla_h u_{\varepsilon}(s)\|_{L^2}^2 ds
\big(\frac{\nu_h}{2} + \frac{1}{\nu_h}\|\partial_3
u_{\varepsilon}(0)\|_{L^2}^2 \big) < \frac{1}{2}.
\end{equation}
Actually, the property \eqref{bornexist0} will be proved if we show that,
for any positive number $\delta$, there
exist two positive numbers $T_0=T_0(\delta)$ and
$\eta_0=\eta_0(\delta)$ such that
\begin{equation}
\label{bornexist}
\int_{0}^{T_0} \|\nabla_h u_{\varepsilon}(s)\|_{L^2}^2 ds \leq \delta.
\end{equation}
The remaining part of the proof consists in showing Property
\eqref{bornexist}.  \HB
Notice that Lemma \ref{energie} gives us an estimate of the quantity
$\int_{0}^{t} \|\nabla_h u_{\varepsilon}(s)\|_{L^2}^2 ds$, which we have
used in the proofs of Theorems \ref{Global1} and
\ref{Global2}. Unfortunately, here the initial data
$u_{\varepsilon}=v_0$ are not necessarily small. In order to prove
Property  \eqref{bornexist}, we write the solution $u_{\varepsilon}$
as
$$
u_{\varepsilon}= v_{\varepsilon} + z_{\varepsilon},
$$
where $v_{\varepsilon}$ is the solution of the linear Stokes problem
   \begin{equation*}(LS_{\varepsilon})
\begin{cases}
\partial_t v_{\varepsilon} -\nu_h \Delta_h v_{\varepsilon}
-\varepsilon\partial_{x_3}^2
v_{\varepsilon}=-\nabla q_{\varepsilon}  ~\mbox{in}~ \tilde{Q}, ~t
>0, \\
\div v_{\varepsilon}=0  ~\mbox{ in } \tilde{Q}, ~ t >0,  \\
v_{\varepsilon}|_{\partial\Omega\times (-1,1)}=0 , ~ t >0, \\
v_{\varepsilon}(x_h,x_3)=v_{\varepsilon}(x_h,x_3+2) , ~ t >0, \\
v_{\varepsilon}|_{t=0} = v_0,
\end{cases}
\end{equation*}
and where $z_{\varepsilon}$ is the solution of the following
auxiliary nonlinear system
   \begin{equation*}(Z_{\varepsilon})
\begin{cases}
\partial_t z_{\varepsilon} + (z_{\varepsilon}+ v_{\varepsilon})\nabla
(z_{\varepsilon} + v_{\varepsilon}) -\nu_h \Delta_h z_{\varepsilon}
-\varepsilon\partial_{x_3}^2
z_{\varepsilon}=-\nabla q^*_{\varepsilon} ~\mbox{in} ~\tilde{Q},\,t
>0, \\
\div z_{\varepsilon}=0  ~\mbox{ in } \tilde{Q}, ~t >0,  \\
z_{\varepsilon}|_{\partial\Omega\times (-1,1)}=0 ~, ~t >0, \\
z_{\varepsilon}(x_h,x_3)=z_{\varepsilon}(x_h,x_3+2),~t >0, \\
z_{\varepsilon}|_{t=0} = 0.
\end{cases}
\end{equation*}
The Stokes problem $(LS_{\varepsilon})$ admits a unique (global)
classical solution $v_{\varepsilon}$ in $C^0([0, +\infty),
\tilde{V})$. Lemma \ref{energie} implies that, for any $t \geq 0$,
\begin{equation}
\label{vestimeaux1}
\|v_{\varepsilon}(t)\|_{L^{2}}^{2} + \nu_h \int_{0}^{t} \|\nabla_h 
v_\varepsilon(s)\|_{L^2}^{2} ds +
\varepsilon \int_{0}^{t} \|\partial_3 v_\varepsilon(s)\|_{L^2}^{2}
ds \leq   \|v_0\|_{L^2}^{2}.
\end{equation}
Arguing as in the proofs of Theorems \ref{Global1} and
\ref{Global2}, one at once shows that, for $t \geq 0$,
\begin{equation}
\label{vestimeaux2}
\|\partial_3v_{\varepsilon}(t)\|_{L^{2}}^{2} + \nu_h \int_{0}^{t}
\|\nabla_h \partial_3v_\varepsilon(s)\|_{L^2}^{2} ds +
\varepsilon \int_{0}^{t} \|\partial_3^2 v_\varepsilon(s)\|_{L^2}^{2}
ds \leq   \|\partial_3v_0\|_{L^2}^{2}.
\end{equation}
Notice that Problem $(Z_{\varepsilon})$ also admits a unique
classical solution $z_{\varepsilon} \in C^0([0, T_{\varepsilon}),
\tilde{V})$, where $T_{\varepsilon}$ is the maximal time of existence
of $u_{\varepsilon}$.

We will prove that, for $\eta >0$ small enough,
there exists $T_0 >0$, independent of $\varepsilon$, but depending on
$U_0$, such that,
\begin{equation}
\label{bornexist2}
\int_{0}^{T_0} \|\nabla_h v_{\varepsilon}(s)\|_{L^2}^2 ds \leq
\frac{\delta}{2}~, \quad \int_{0}^{T_0} \|\nabla_h 
z_{\varepsilon}(s)\|_{L^2}^2 ds \leq
\frac{\delta}{2}~.
\end{equation}
\vskip 1mm

We introduce a positive number $\delta_0 \leq \delta$, which will be
made more precise later.
In order to prove the first inequality  of \eqref{bornexist2}, we
proceed as follows by decomposing the linear Stokes problem
$(LS_{\varepsilon})$ into two
auxiliary linear systems, the first one with very regular initial
data and the second one with small initial data. We recall that $P$
is the classical Leray projector. Let $A_0$ be the
Stokes operator $A_0= - P\Delta$ with homogeneous Dirichlet boundary
conditions on $\partial \Omega \times (-1,1)$ and
periodic boundary conditions in the vertical variable. The spectrum
of $A_0$ consists in a nondecreasing sequence of eigenvalues
$$
0 < \lambda_0 < \lambda_1 \leq \lambda_2 \leq \cdots \leq \lambda_m
\leq \cdots,
$$
going to infinity as $m$ goes to infinity. We denote
$\P_k$ the projection onto the space generated by the eigenfunctions
associated to the first $k$ eigenvalues of the operator $A_0$. There
exists an integer $k_0=k_0(\delta_0)$ such that,
\begin{equation}
\label{lambdak0}
\|(I-\P_k) \Sigma U_0\|_{L^{2}}^2 \leq \frac{\delta_0 \nu_h}{16}~, \forall
k\geq k_0~.
\end{equation}
If $v_0 \in H^1_{0, per}(\tilde{Q})$ satisfies the condition
\eqref{v0U0}, the property \eqref{lambdak0} implies that
\begin{equation}
\label{lambdak0bis}
\|(I-\P_{k_0}) v_0\|_{L^{2}}^2 \leq  2 \|(I-\P_{k_0})
U_0\|_{L^{2}}^2 + 2 \|(I-\P_{k_0}) (U_0-v_0)\|_{L^{2}}^2 \leq
\frac{\delta_0 \nu_h}{8} + 2 \eta^2~.
\end{equation}
On the other hand, the following obvious estimate holds,
\begin{equation}
\label{estH1k0}
\|\P_{k_0} v_0\|_{H^1(\tilde{Q})}^2 \leq \lambda_{k_0 +1}
\|v_0\|_{L^{2}}^2 \leq
\lambda_{k_0 +1} ( 2 \|U_0\|_{L^{2}}^2 + 2 \eta^2)~.
\end{equation}
We next decompose $v_{\varepsilon}$ into the sum $v_{\varepsilon}=
v_{1,\varepsilon} + v_{2,\varepsilon}$ where $v_{1,\varepsilon}$ is
the solution of the Stokes problem
   \begin{equation*}(LS_{1,\varepsilon})
\begin{cases}
\partial_t v_{1,\varepsilon} -\nu_h \Delta_h v_{1,\varepsilon}
-\varepsilon\partial_{x_3}^2
v_{1, \varepsilon}=-\nabla q_{1,\varepsilon}  ~\mbox{ in } \tilde{Q}~, ~t
>0, \\
\div v_{1,\varepsilon}=0  ~\mbox{ in } \tilde{Q}~, ~ t >0,  \\
v_{1,\varepsilon}|_{\partial\Omega\times (-1,1)}=0 ~, ~ t >0, \\
v_{1,\varepsilon}(x_h,x_3)=v_{1,\varepsilon}(x_h,x_3+2) ~, ~ t >0, \\
v_{1,\varepsilon}|_{t=0} = \P_{k_0} v_0,
\end{cases}
\end{equation*}
and $v_{2,\varepsilon}$ is
the solution of the Stokes problem
   \begin{equation*}(LS_{2,\varepsilon})
\begin{cases}
\partial_t v_{2,\varepsilon} -\nu_h \Delta_h v_{2,\varepsilon}
-\varepsilon\partial_{x_3}^2
v_{2,\varepsilon}=-\nabla q_{2,\varepsilon}  ~\mbox{ in } \tilde{Q}~, ~t
>0, \\
\div v_{2,\varepsilon}=0  ~\mbox{ in } \tilde{Q}~, ~ t >0,  \\
v_{2,\varepsilon}|_{\partial\Omega\times (-1,1)}=0 ~, ~ t >0, \\
v_{2,\varepsilon}(x_h,x_3)=v_{2,\varepsilon}(x_h,x_3+2) ~, ~ t >0, \\
v_{2,\varepsilon}|_{t=0} =(I-\P_{k_0}) v_0.
\end{cases}
\end{equation*}
{}From Lemma \ref{energie}, we at once deduce that, for $t \geq 0$,
\begin{equation}
\label{v2lemme4}
\begin{split}
\|v_{2,\varepsilon}(t)\|_{L^{2}}^{2} + \nu_h \int_{0}^{t}
\|\nabla_h v_{2,\varepsilon}(s)\|_{L^2}^{2} ds +
\varepsilon  &\int_{0}^{t} \|\partial_3 v_{2,\varepsilon}(s)\|_{L^2}^{2}
ds \cr
\leq  & \|(I- \P_{k_0})v_0\|_{L^2}^{2}
\leq \frac{\delta_0 \nu_h}{8} + 2\eta^2 ,
\end{split}
\end{equation}
Hence, if $\eta^2 \leq \delta_0 \nu_h /16$, we obtain, for any $t \geq
0$,
\begin{equation}
\label{v2estime}
\int_{0}^{t}
\|\nabla_h v_{2,\varepsilon}(s)\|_{L^2}^{2} ds \leq \frac{\delta_0}{8}
+ \frac{2 \eta^2}{\nu_h}  \leq \frac{\delta_0}{4}.
\end{equation}
In order to get an upper bound of the term $\int_{0}^{t} \|\nabla_h
v_{1,\varepsilon}(s)\|_{L^2}^{2} ds$, we first estimate $ \|\nabla_h
v_{1,\varepsilon}(s)\|_{L^2}^{2}$ for any $s \geq 0$. Like in the
proof of Proposition \ref{gradh}, we take the inner product in
$L^2(\tilde{Q})^3$ of the first equation of $(LS_{1,\varepsilon})$ with
$-P \Delta_h v_{1,\varepsilon}$. Arguing as in the proof of
Proposition \ref{gradh}, we obtain, for any $t \geq 0$,
$$
\partial_t \|\nabla_h v_{1,\varepsilon}\|_{L^{2}}^{2} +  \nu_h \|P
\Delta_h v_{1,\varepsilon}\|_{L^{2}}^{2} +
2 \varepsilon \|\nabla_h \partial_3 v_{1,\varepsilon}\|_{L^{2}}^{2}
\leq 0 .
$$
Integrating the above inequality between $0$ and $t$ and taking into
account the estimate \eqref{estH1k0}, we obtain, for $t \geq 0$,
\begin{equation*}
\begin{split}
\|\nabla_h v_{1,\varepsilon}(t)\|_{L^{2}}^{2} +  \nu_h \int_{0}^{t}
  \|P \Delta_h &v_{1,\varepsilon}(s)\|_{L^{2}}^{2} ds +
2 \varepsilon \int_{0}^{t} \|\nabla_h \partial_3
v_{1,\varepsilon}(s)\|_{L^{2}}^{2}ds \cr
\leq & \|\P_{k_0}v_0\|_{H^1(\tilde{Q})}^{2} \leq \lambda_{k_0 +1}(2
\|U_0\|_{L^{2}}^{2} + 2\eta^2).
\end{split}
\end{equation*}
{}From the above inequality, we deduce that, for $t \geq 0$,
\begin{equation}
\label{v1estime1}
\int_{0}^{t}  \|\nabla_h v_{1,\varepsilon}(s)\|_{L^{2}}^{2} ds \leq t
\lambda_{k_0 +1}(2 \|U_0\|_{L^{2}}^{2} + 2\eta^2) \leq
t\lambda_{k_0 +1}(2 \|U_0\|_{L^{2}}^{2} +  \frac{\delta_0 \nu_h}{8}) ,
\end{equation}
  and thus, if
  \begin{equation}
  \label{T0cond1}
  0 < T_0 \leq  \frac{\delta_0}{4 \lambda_{k_0 +1}}(2 \|U_0\|_{L^{2}}^{2} +
\frac{\delta_0 \nu_h}{8})^{-1},
\end{equation}
we have
\begin{equation}
\label{v1estime2}
\int_{0}^{T_0}  \|\nabla_h v_{1,\varepsilon}(s)\|_{L^{2}}^{2} ds
\leq \frac{\delta_0}{4}  .
\end{equation}
The inequalities \eqref{v2estime} and \eqref{v1estime2} imply that,
if $\eta^2 \leq \delta_0 \nu_h /16$ and if the condition
\eqref{T0cond1} holds, then
\begin{equation}
\label{vborneinteg}
\int_{0}^{T_0}  \|\nabla_h v_{\varepsilon}(s)\|_{L^{2}}^{2} ds
\leq \frac{\delta_0}{2}  \leq \frac{\delta}{2}.
\end{equation}
\vskip 1mm

It remains to bound the integral $\int_{0}^{T_0}  \|\nabla_h
z_{\varepsilon}(s)\|_{L^{2}}^{2} ds$. Taking the inner product in
$L^2(\tilde{Q})^3$ of the first equation of System
$(Z_{\varepsilon})$ with $z_{\varepsilon}$, we obtain the equality
\begin{equation}
\label{zepsaux1}
\begin{split}
\frac{1}{2} \partial_t \| z_{\varepsilon}\|_{L^{2}}^{2} +  \nu_h
\|\nabla_h z_{\varepsilon}\|_{L^{2}}^{2} +
  \varepsilon \|\partial_3 z_{\varepsilon}\|_{L^{2}}^{2}
=& -(z_{\varepsilon,3} \partial_3 v_{\varepsilon},z_{\varepsilon})
-(z_{\varepsilon,h} \nabla_h v_{\varepsilon},z_{\varepsilon})  \cr
& -(v_{\varepsilon,3} \partial_3 v_{\varepsilon},z_{\varepsilon})
-(v_{\varepsilon,h} \nabla_h v_{\varepsilon},z_{\varepsilon})
\end{split}
\end{equation}
We next estimate the four terms of the right-hand side member of the equality
\eqref{zepsaux1}. Applying Lemma \ref{LestLpLq} and using the fact
that $\partial_3 z_{\varepsilon,3}=- \div_h z_{\varepsilon,h}$, we
obtain, for $0 \leq t\leq T_{\varepsilon}$,
\begin{equation*}
\begin{split}
|(z_{\varepsilon,3} \partial_3 v_{\varepsilon},z_{\varepsilon})| \leq
\, &  \|z_{\varepsilon,3}\|_{L^{\infty}_v(L^2_h)} \|\partial_3
v_{\varepsilon}\|_{L^{2}_v(L^4_h)}
\|z_{\varepsilon}\|_{L^{2}_v(L^4_h)} \cr
\leq \,  &C_0^3 \big(\|z_{\varepsilon,3}\|_{L^2}^{1/2} \|\partial_3
z_{\varepsilon,3}\|_{L^2}^{1/2} + \|z_{\varepsilon,3}\|_{L^2} \big) \cr
& \times \|\partial_3 v_{\varepsilon}\|_{L^2}^{1/2}
\|\partial_3 \nabla_h v_{\varepsilon}\|_{L^2}^{1/2}
\|z_{\varepsilon}\|_{L^2}^{1/2} \|\nabla_h
z_{\varepsilon}\|_{L^2}^{1/2} \cr
\leq \,&C_0^3 \big(\|z_{\varepsilon,3}\|_{L^2}^{1/2} \|\nabla_h
z_{\varepsilon,h}\|_{L^2}^{1/2} + \|z_{\varepsilon,3}\|_{L^2}\big) \cr
& \times \|\partial_3 v_{\varepsilon}\|_{L^2}^{1/2}
\|\partial_3 \nabla_h v_{\varepsilon}\|_{L^2}^{1/2}
\|z_{\varepsilon}\|_{L^2}^{1/2} \|\nabla_h
z_{\varepsilon}\|_{L^2}^{1/2} .
\end{split}
\end{equation*}
Applying the Young inequality to the above estimate, we get the
inequality
\begin{equation}
\label{zepsaux2}
\begin{split}
|(z_{\varepsilon,3} \partial_3 v_{\varepsilon},z_{\varepsilon})| \leq
& \frac{\nu_h}{8}\|\nabla_h z_{\varepsilon}\|_{L^2}^{2}  +
\frac{4C_0^6}{\nu_h} \|\partial_3 v_{\varepsilon}\|_{L^2}
\|\partial_3 \nabla_h v_{\varepsilon}\|_{L^2}
\|z_{\varepsilon}\|_{L^2}^{2} \cr
&+ \frac{3C_0^4}{2 \nu_h^{1/3}} \|\partial_3 v_{\varepsilon}\|_{L^2}^{2/3}
\|\partial_3 \nabla_h v_{\varepsilon}\|_{L^2}^{2/3} 
\|z_{\varepsilon}\|_{L^2}^{2}.
\end{split}
\end{equation}
Applying Lemma \ref{LestLpLq} and the Young inequality again, we also obtain
the following estimate, for $0 \leq t\leq T_{\varepsilon}$,
\begin{equation}
\label{zepsaux3}
\begin{split}
|(z_{\varepsilon,h} \nabla_h v_{\varepsilon},z_{\varepsilon})| & \leq
\, C_0^2 \|\nabla_h v_{\varepsilon}\|_{L^{\infty}_v(L^2_h)} \|\nabla_h
z_{\varepsilon}\|_{L^{2}} \|z_{\varepsilon}\|_{L^{2}} \cr
\leq \,  &C_0^3 \Big( \|\nabla_h v_{\varepsilon}\|_{L^2}^{1/2}
\|\nabla_h \partial_3 v_{\varepsilon}\|_{L^2}^{1/2}
+ \|\nabla_h v_{\varepsilon}\|_{L^2} \Big)
\|\nabla_h z_{\varepsilon}\|_{L^{2}} \|z_{\varepsilon}\|_{L^{2}} \cr
\leq \, & \frac{\nu_h}{8}\|\nabla_h z_{\varepsilon}\|_{L^2}^{2}
+\frac{C_0^6}{\nu_h}  \|z_{\varepsilon}\|_{L^{2}}^2 \Big(9 \|\nabla_h 
v_{\varepsilon}\|_{L^2}^2
+ \|\nabla_h \partial_3 v_{\varepsilon}\|_{L^2}^2 \Big).
\end{split}
\end{equation}
Applying Lemma \ref{LestLpLq} again, we
can write, for $0 \leq t\leq T_{\varepsilon}$,
\begin{equation*}
\begin{split}
|(v_{\varepsilon,3} \partial_3 v_{\varepsilon},z_{\varepsilon})| \leq
\, &  \|v_{\varepsilon,3}\|_{L^{\infty}_v(L^2_h)} \|\partial_3
v_{\varepsilon}\|_{L^{2}_v(L^4_h)}
\|z_{\varepsilon}\|_{L^{2}_v(L^4_h)} \cr
\leq \,&C_0^3 \big(\|v_{\varepsilon,3}\|_{L^2}^{1/2} \|\partial_3
v_{\varepsilon,3}\|_{L^2}^{1/2} + \|v_{\varepsilon,3}\|_{L^2}\big) \cr
& \times \|\partial_3 v_{\varepsilon}\|_{L^2}^{1/2}
\|\partial_3 \nabla_h v_{\varepsilon}\|_{L^2}^{1/2}
\|z_{\varepsilon}\|_{L^2}^{1/2} \|\nabla_h
z_{\varepsilon}\|_{L^2}^{1/2} .
\end{split}
\end{equation*}
Using the Young inequality several times, we deduce from the above
estimate that
\begin{equation*}
\begin{split}
|(v_{\varepsilon,3} &\partial_3 v_{\varepsilon},z_{\varepsilon})| \cr
\leq  \,& \frac{3C_0^4}{2\nu_h^{1/3}} \Big( \|v_{\varepsilon}\|_{L^2}^{2/3}
\|\partial_3 v_{\varepsilon}\|_{L^2}^{2/3} +
\|v_{\varepsilon}\|_{L^2}^{4/3}\Big)
\|\partial_3 v_{\varepsilon}\|_{L^2}^{2/3}
\|\partial_3 \nabla_h v_{\varepsilon}\|_{L^2}^{2/3}
\|z_{\varepsilon}\|_{L^2}^{2/3} \cr
&+\frac{\nu_h}{8}\|\nabla_h z_{\varepsilon}\|_{L^2}^{2},
\end{split}
\end{equation*}
and thus that
\begin{equation}
\label{zepsaux4}
\begin{split}
|(v_{\varepsilon,3} \partial_3 v_{\varepsilon},z_{\varepsilon})|
\leq  \,  & \frac{\nu_h}{8}\|\nabla_h z_{\varepsilon}\|_{L^2}^{2}
+  \frac{C_0^4}{\nu_h} \|\partial_3 \nabla_h v_{\varepsilon}\|_{L^2}^{2}
\|z_{\varepsilon}\|_{L^2}^{2} + C_0^4
\|v_{\varepsilon}\|_{L^2}^{2} \|\partial_3 v_{\varepsilon}\|_{L^2} \cr
&+ C_0^4
\|v_{\varepsilon}\|_{L^2}\|\partial_3 v_{\varepsilon}\|_{L^2}^2 .
\end{split}
\end{equation}
Finally, arguing as above by applying Lemma \ref{LestLpLq} and the
Young inequality, we get the estimate
\begin{equation}
\label{zepsaux5}
\begin{split}
|(v_{\varepsilon,h} \nabla_h v_{\varepsilon},z_{\varepsilon})|  \leq
\,  & C_0^2\|\nabla_h v_{\varepsilon}\|_{L^{\infty}_v(L^2_h)} \|\nabla_h
v_{\varepsilon}\|_{L^{2}}^{1/2} \|v_{\varepsilon}\|_{L^{2}}^{1/2}
\|\nabla_h z_{\varepsilon}\|_{L^{2}}^{1/2}
\|z_{\varepsilon}\|_{L^{2}}^{1/2} \cr
\leq \,  &C_0^3 \Big( \|\nabla_h v_{\varepsilon}\|_{L^2}^{1/2}
\|\nabla_h \partial_3 v_{\varepsilon}\|_{L^2}^{1/2}
+ \|\nabla_h v_{\varepsilon}\|_{L^2} \Big) \cr
& \times \|\nabla_h
v_{\varepsilon}\|_{L^{2}}^{1/2} \|v_{\varepsilon}\|_{L^{2}}^{1/2}
\|\nabla_h z_{\varepsilon}\|_{L^{2}}^{1/2}
\|z_{\varepsilon}\|_{L^{2}}^{1/2} \cr
\leq \, & \frac{\nu_h}{8}\|\nabla_h z_{\varepsilon}\|_{L^2}^{2}
+ 2C_0^4 \|\nabla_h v_{\varepsilon}\|_{L^2}^2
\|v_{\varepsilon}\|_{L^2} \cr
&+ \frac{C_0^4}{2\nu_h} \|z_{\varepsilon}\|_{L^{2}}^{2} \Big(
\|\nabla_h v_{\varepsilon}\|_{L^2}^2 + \|\nabla_h \partial_3
v_{\varepsilon}\|_{L^2}^2\Big).
\end{split}
\end{equation}
Integrating the equality \eqref{zepsaux1} from $0$ to $t$, taking
into account the estimates \eqref{zepsaux2} to
\eqref{zepsaux5} and, applying Gronwall lemma yield, for $0 \leq t 
\leq T_{\varepsilon}$,
\begin{equation*}
\begin{split}
  \| z_{\varepsilon}(t)\|_{L^{2}}^{2} +  \nu_h \int_{0}^{t}
\|\nabla_h z_{\varepsilon}(s)&\|_{L^{2}}^{2} ds  +
2 \varepsilon \int_{0}^{t} \|\partial_3
  z_{\varepsilon}(s)\|_{L^{2}}^{2} ds \cr
   & \leq \int_{0}^{t} B_1(s)
  \|z_{\varepsilon}(s)\|_{L^{2}}^{2}ds +
  \int_{0}^{t} B_2(s) ds,
\end{split}
\end{equation*}
and
\begin{equation}
\label{zepsaux6}
  \| z_{\varepsilon}(t)\|_{L^{2}}^{2} +  \nu_h \int_{0}^{t}
\|\nabla_h z_{\varepsilon}(s)\|_{L^{2}}^{2} ds \leq 2 \int_{0}^{t}
B_2(s) ds \,\big( \exp \,2\int_{0}^{t} B_1(s) ds \big) ,
  \end{equation}
where
\begin{equation}
\label{zepsaux7}
\begin{split}
B_1(s)=&2 \Big[\frac{C_0^4}{\nu_h} (3C_0^2 +2) \|\nabla_h \partial_3
v_{\varepsilon}(s)\|_{L^2}^2
+  \frac{C_0^4}{2\nu_h} (18 C_0^2 +1) \|\nabla_h
v_{\varepsilon}(s)\|_{L^2}^2 \cr
&+ \frac{2 C_0^6}{\nu_h}  \| \partial_3
v_{\varepsilon}(s)\|_{L^2}^2 + C_0^4\| \partial_3
v_{\varepsilon}(s)\|_{L^2} \cr
B_2(s) =& \, 2C_0^4\Big( \| v_{\varepsilon}(s)\|_{L^2}^2
\| \partial_3 v_{\varepsilon}(s)\|_{L^2}
+\| v_{\varepsilon}(s)\|_{L^2}\| \partial_3
v_{\varepsilon}(s)\|_{L^2}^2  \cr
& +2 \| v_{\varepsilon}(s)\|_{L^2} \| \nabla_h
v_{\varepsilon}(s)\|_{L^2}^2 \Big).
\end{split}
\end{equation}
The inequalities \eqref{zepsaux6} and \eqref{zepsaux7} and the
estimates \eqref{vestimeaux1},  \eqref{vestimeaux2}, and
\eqref{vborneinteg} imply that, for $0 \leq t \leq T_{\varepsilon}$,
\begin{equation}
\label{zepsaux8}
\begin{split}
\int_{0}^{t} \|\nabla_h z_{\varepsilon}(s)\|_{L^{2}}^{2} ds
\leq &\frac{c_1}{\nu_h}\| v_0\|_{L^2} \Big( t\| v_0\|_{L^2} \|
\partial_3 v_0\|_{L^2} + t \| \partial_3 v_0\|_{L^2}^2 + \delta_0
\Big) \cr
\times &\exp c_2 \Big( t \| \partial_3 v_0\|_{L^2} +
\nu_h^{-1} \| \partial_3 v_0\|_{L^2}^2 (t +\nu_h^{-1}) + \nu_h^{-1}
\delta_0\Big) ,
\end{split}
\end{equation}
where $c_1$ and $c_2$ are two positive constants independent of $v_0$
and $\varepsilon$. Since
$$
\|v_0\|_{H^{0,1}(\tilde{Q})} \leq  \|\Sigma
U_0\|_{H^{0,1}(\tilde{Q})} + \eta\leq  \|\Sigma
U_0\|_{H^{0,1}(\tilde{Q})} + \frac{(\delta_0 \nu_h)^{1/2}}{4},
$$
the inequality \eqref{zepsaux8} shows that we can choose $\delta_0
>0$ and $T_0>0$ independent of $\varepsilon$ and $v_0$ such that
\begin{equation}
\label{zepsaux9}
\begin{split}
\frac{c_1}{\nu_h}\| v_0\|_{L^2} &\Big( T_0\| v_0\|_{L^2} \|
\partial_3 v_0\|_{L^2} + T_0\| \partial_3 v_0\|_{L^2}^2 + \delta_0
\Big) \cr
\times &\exp c_2 \Big( T_0\| \partial_3 v_0\|_{L^2} +
\nu_h^{-1} \| \partial_3 v_0\|_{L^2}^2 (T_0 +\nu_h^{-1}) + \nu_h^{-1}
\delta_0\Big)\leq \frac{\delta} {2}.
\end{split}
\end{equation}
As we have explained at the beginning of the proof, the properties
\eqref{zepsaux8} and \eqref{zepsaux9} imply that the maximal time
$T_{\varepsilon}$ of existence of $z_{\varepsilon}$ and of
$u_{\varepsilon}$ is larger than $T_0$ and that
$$
\int_{0}^{T_0} \|\nabla_h z_{\varepsilon}(s)\|_{L^{2}}^{2} ds \leq 
\frac{\delta} {2}.
$$
Thus the inequalities \eqref{bornexist2} are proved, which concludes
the proof of the theorem.
\end{proof}

\begin{remark} We notice that the classical approach to show the
local in time existence result for large initial data, consisting in
the decompostion
of the problem into a large data linear problem and a small data,
perturbed nonlinear problem, does not work here, since we cannot
prove that, for initial data $U_0$ in $\widetilde{H}^{0,1}(Q)$, the
quantity $\|(I- \P_{k_0})v_0\|_{H^{0,1}(\tilde{Q})}$ is small. In the
above proof, the decomposition of the linear system into two systems,
one with smooth initial data $\P_{k_0}v_0$ and the other one with
small initial data $(I- \P_{k_0})v_0$ avoids this difficulty. Indeed,
in the
estimates \eqref{v2lemme4} and \eqref{v2estime}, we only need to know
that $\|(I- \P_{k_0})v_0\|_{L^2(\tilde{Q})}$ is small.
\end{remark}




\end{document}